\documentclass[11pt,a4paper]{article}

\usepackage[margin=1in]{geometry}
\usepackage[T1]{fontenc}
\usepackage[utf8]{inputenc}
\usepackage{lmodern}
\usepackage{microtype}
\usepackage{xcolor}
\definecolor{revisionpurple}{RGB}{128,0,128}
\usepackage{amsmath,amssymb,amsfonts,amsthm,mathtools,bm}
\usepackage{enumitem}
\usepackage{booktabs}
\usepackage[numbers,sort&compress]{natbib}
\usepackage[colorlinks=true,linkcolor=blue,citecolor=blue,urlcolor=blue]{hyperref}
\usepackage[nameinlink,capitalize,noabbrev]{cleveref}
\newcommand{\doi}[1]{\href{https://doi.org/#1}{\texttt{#1}}}
\newcommand{\arxivdoi}[1]{\href{https://doi.org/10.48550/arXiv.#1}{\texttt{10.48550/arXiv.#1}}}

\setlist[itemize]{leftmargin=2em}
\setlist[enumerate]{leftmargin=2em}

\newtheorem{theorem}{Theorem}[section]
\newtheorem{proposition}[theorem]{Proposition}
\newtheorem{lemma}[theorem]{Lemma}
\newtheorem{corollary}[theorem]{Corollary}
\newtheorem{assumption}[theorem]{Assumption}
\newtheorem{definition}[theorem]{Definition}
\newtheorem{remark}[theorem]{Remark}
\newtheorem{example}[theorem]{Example}

\newcommand{\R}{\mathbb{R}}
\newcommand{\N}{\mathbb{N}}
\newcommand{\E}{\mathbb{E}}
\newcommand{\Prob}{\mathbb{P}}
\newcommand{\calP}{\mathcal{P}}
\newcommand{\calL}{\mathcal{L}}

\newcommand{\calA}{\mathcal{A}}
\newcommand{\calU}{\mathcal{U}}
\newcommand{\dd}{\,\mathrm{d}}

\newcommand{\softmax}{\mathrm{softmax}}
\newcommand{\rowsoftmax}{\mathrm{softmax}_{\mathrm{row}}}
\newcommand{\inner}[2]{\left\langle #1,#2\right\rangle}
\newcommand{\norm}[1]{\left\lVert #1\right\rVert}
\newcommand{\abs}[1]{\left\lvert #1\right\rvert}

\newcommand{\mcJ}{\mathcal{J}}

\newcommand{\eps}{\varepsilon}
\newcommand{\thetaeps}{\theta^{\eps}}

\newcommand{\DeltaV}{\Delta_{d_v}}

\title{\bfseries A First-Order Mean Field Control Analysis of Transformer Layers under Cross-Entropy Training}

\makeatletter

\def\author#1#2{\expandafter\gdef\csname transformer@author#1\endcsname{#2}\transformer@setauthors}
\def\address#1#2{\expandafter\gdef\csname transformer@address#1\endcsname{#2}\transformer@setauthors}
\def\email#1#2{\expandafter\gdef\csname transformer@email#1\endcsname{#2}\transformer@setauthors}
\newcommand{\transformer@authorentry}[1]{%
  \csname transformer@author#1\endcsname\thanks{%
  \csname transformer@address#1\endcsname. Email: \href{mailto:\csname transformer@email#1\endcsname}{\texttt{\csname transformer@email#1\endcsname}}}}
\newcommand{\transformer@setauthors}{%
  \gdef\@author{%
    \@ifundefined{transformer@author1}{}{\transformer@authorentry{1}}%
    \@ifundefined{transformer@author2}{}{\and\transformer@authorentry{2}}%
  }%
}
\makeatother
\author1{Cheng Huan}
\address1{Department of Statistics and Data Science, The Chinese University of Hong Kong}
\email1{chenghuan@cuhk.edu.hk}
\author2{Hongwei Yuan}
\address2{Department of Mathematics, University of Macau}
\email2{hwyuan@um.edu.mo}
\date{}

\begin{document}
\maketitle

\begin{abstract}
We study Transformer-type residual layers under cross-entropy training through a continuous-depth mean field control viewpoint.  Depth is treated as time, layer parameters as controls, and the residual Transformer recursion as an explicit Euler scheme for a controlled hidden-state flow.  For fixed controls, we prove an $O(\varepsilon)$ pathwise approximation of finite-depth trajectories by the continuous flow and combine this with high-probability sampling bounds for the empirical cross-entropy risk.  We formulate the limiting population problem as a first-order transport control problem for the law of hidden states and derive a Pontryagin condition whose terminal adjoint contains the softmax residual.  We also give finite-class and metric-entropy uniform estimates, compare optimal values, and discuss existence, stability, continuous-to-discrete recovery, initialization, and range estimates for continuous minimizers.
\end{abstract}



\section{Introduction}
The Transformer architecture of \citet{Vaswani2017} has become the dominant model class for sequence modeling, language modeling, and many representation-learning tasks.  Its success has also motivated a growing mathematical literature on attention mechanisms, normalization, residual information flow, and interacting-particle interpretations of token dynamics; see, for example, \citet{GeshkovskiLetrouitPolyanskiyRigollet2025}.  A Transformer layer updates a sequence of token representations through a residual operation involving self-attention, feed-forward transformations, and normalization.  In many mathematical discussions it is useful to isolate this residual character and write a layer in the form
\begin{align}
\label{eq:intro-transformer-layer}
    X_{k+1}=X_k+h_k F(X_k;\theta_k),
\end{align}
where $X_k\in\R^{L\times d_x}$ is the sequence of $L$ token embeddings at depth $k$, $\theta_k$ is the parameter of the $k$-th Transformer block, and $F:\R^{L\times d_x}\times\Theta\to\R^{L\times d_x}$ is a self-attention-type vector field.  This residual form suggests a continuous-depth interpretation: when the layer step size is small and the number of layers is large, the hidden state should be close to the solution of a controlled ordinary differential equation.  This viewpoint is closely related to the dynamical-systems interpretation of residual networks \citep{E2017,HaberRuthotto2017,LiShi2017}, to neural ordinary differential equations \citep{ChenRubanovaBettencourtDuvenaud2018}, and to PDE-inspired deep network models \citep{RuthottoHaber2020}.

The main motivation of this work is to turn the continuous-depth intuition for residual Transformers into a quantitative approximation and control framework for population training.  A finite Transformer is trained simultaneously on many input sequences, so its layer parameters act as a common control shared by all samples rather than as controls tailored to a single trajectory.  Tracking the induced law of hidden states across depth therefore leads naturally to a deterministic first-order mean field control problem in which the state variable is the population distribution of representations and the terminal cost is the cross-entropy risk.  This point of view is consistent with mean field type control theory \citep{CarmonaDelarue2018,BensoussanWongYamYuan2026}, but it is adapted here to residual self-attention dynamics and to token-level classification losses.

The control formulation is useful for three reasons.  First, it separates the depth-discretization error, the statistical sampling error, and the optimization error over an admissible control class, which are often intertwined in finite-depth training.  Second, it gives a precise variational structure: the limiting problem admits a transport equation and a Pontryagin condition whose terminal adjoint is the softmax residual generated by cross-entropy.  Third, it clarifies which later stability and convergence statements are conditional on compactness, local coercivity, or PL-type assumptions, rather than consequences of the Transformer architecture alone.

We consider supervised sequence-label data in which the input sequence $X_0$ has law $\mu_X$ on $\R^{L\times d_x}$ and the labels are deterministic next-token labels generated from $X_0$.  More precisely, $Y_l(X_0)\in\{e_1,\ldots,e_{d_v}\}\subset\R^{d_v}$ is the one-hot next-token label at position $l\in\{0,\ldots,L-1\}$, completely determined by $X_0$.  We write $Y(X_0)=\bigl(Y_0(X_0),\ldots,Y_{L-1}(X_0)\bigr)^\top$.  The finite-depth Transformer recursion is modeled as
\begin{equation}
\label{eq:intro-discrete}
    X_{k+1}^{\eps}=X_k^{\eps}+\eps\zeta(k\eps)F(X_k^{\eps};\theta_k),
    \qquad k=0,\ldots,M-1,
\end{equation}
where $M\eps=T$ is the final depth, $\zeta$ is a bounded depth profile, and $\theta_k$ is the trainable layer parameter.  The final logits are obtained from a readout map $H:\R^{L\times d_x}\to\R^{L\times d_v}$, and the training loss is the token-averaged cross-entropy
\begin{align}
\label{eq:intro-ce}
    \ell(Y(X_0),H(X))=-\frac1L\sum_{l=0}^{L-1}\sum_{j=1}^{d_v}Y_{l,j}(X_0)\log \softmax_j(H(X)_{l,:}).
\end{align}
When $\eps\downarrow0$, the recursion \eqref{eq:intro-discrete} is an explicit Euler discretization of the controlled Transformer flow
\begin{equation}
\label{eq:intro-ode}
    \dot X_t=\zeta(t)F(X_t;\theta_t),
    \qquad X_0\sim \mu_X.
\end{equation}
The special case $\zeta\equiv 1$ gives the simpler equation $\dot X_t=F(X_t;\theta_t)$.  We keep $\zeta$ explicitly since it is useful for tracking nonuniform depth weights and for comparison with layerwise learning-rate scalings.  It may also be interpreted as a scalar continuous-depth analogue of a residual-stream strength or gate.  This interpretation is related in spirit to hyper-connections and manifold-constrained hyper-connections, which modify residual information flow by expanding, mixing, or constraining connection paths across depth \citep{XieWeiCaoZhaoDengLiDaiGaoChangYuZhaoZhouXuZhangZengHuWangYuanWangLiang2025}.  The present $\zeta(t)$, however, is only a scalar modulation of a single residual vector field; it does not model the full hyper-connection mechanism, where several residual streams are mixed by structured connection operators.

The population continuous-depth objective is
\begin{equation}
\label{eq:intro-pop-risk}
    \mcJ(\theta)=\E_{X_0\sim\mu_X}\bigl[\ell(Y(X_0),H(X_T^{\theta}(X_0)))\bigr].
\end{equation}
Equivalently, if $\rho_t^\theta=\calL(X_t^\theta)$ denotes only the law of the hidden state $X_t^\theta$, then $\rho_t^\theta$ solves the first-order transport equation
\begin{equation}
\label{eq:intro-transport}
\partial_t\rho_t^\theta+\nabla_X\cdot\bigl(\zeta(t)F(X;\theta_t)\rho_t^\theta\bigr)=0,
    \qquad \rho_0^\theta=\mu_X.
\end{equation}
The label variable is not transported as part of the state.  Since $Y(X_0)$ is fixed by the initial input, the objective is evaluated along characteristics as in \eqref{eq:intro-pop-risk}.  Thus training the continuous-depth Transformer is a deterministic first-order mean field control problem: the control is the layer-parameter curve $t\mapsto\theta_t$, the state is the population law of the hidden representation $X_t^\theta$, and the terminal cost is the cross-entropy risk.  This is complementary to our previous large-head mean-field analysis of attention \citep{HuanYuan2026MeanFieldAttention}: there the empirical law is the law of attention heads in parameter space, whereas here the law is the data-induced population distribution of hidden states along depth. For the mean field viewpoint of interpreting neurons or units as particles in parameter space whose empirical distribution converges to a deterministic measure-valued evolution under suitable scaling, one can refer to \citep{MeiMontanariNguyen2018,SirignanoSpiliopoulos2020,RotskoffVandenEijnden2022}.

We quantify two approximation errors between the original finite training problem and the limiting mean field control problem.  The first is a depth-discretization error: the finite Transformer recursion \eqref{eq:intro-discrete} approximates \eqref{eq:intro-ode}.  The second is a sampling error: the population expectation in \eqref{eq:intro-pop-risk} is replaced by an empirical average over $N$ training samples.  Under boundedness and Lipschitz assumptions, and using standard concentration ideas for bounded independent samples \citep{BoucheronLugosiMassart2013}, our main estimates show that, for each fixed control,
\begin{equation}
\label{eq:intro-rate}
    \abs{\mcJ_{N,\eps}(\thetaeps)-\mcJ(\thetaeps)}
    \leq C\eps+C\sqrt{\frac{\log(1/\delta)}{N}}
\end{equation}
with probability at least $1-\delta$.  We also provide finite-class and metric-entropy uniform variants, which are natural estimates for comparing approximate optimal values.  These uniform statistical estimates are over fixed admissible classes chosen independently of the sample.  Beyond approximation, we discuss how compactness can arise from layer normalization, how dropout-type mechanisms and quadratic regularization relate to stability \citep{HintonEtAl2012,WagerWangLiang2013,WanEtAl2013}, and why nonconvex neural-network landscapes require local rather than global convexity assumptions \citep{GoodfellowVinyalsSaxe2015,ChoromanskaHenaffMathieuArousLeCun2015}.  The local SGD statement uses standard smooth-descent and PL-type reasoning from first-order optimization \citep{Nesterov2018}.

The main contributions are the following.
\begin{enumerate}
    \item We introduce a continuous-depth first-order mean field control formulation for Transformer cross-entropy training.  The transported law is the data-induced law of hidden states along depth, while the labels remain attached to the initial samples and enter only through the terminal Lagrangian loss.
    \item We prove a deterministic Euler approximation theorem for residual Transformer layers: for a fixed layer sequence and a fixed depth horizon, the finite-depth hidden trajectory is within $O(\eps)$ of the corresponding controlled continuous flow.
    \item We combine this discretization estimate with concentration for bounded cross-entropy losses.  For each fixed control, the empirical finite-depth risk differs from the population continuous-depth risk by an $O(\eps)$ deterministic term plus a high-probability $O(N^{-1/2})$ sampling term.
    \item We give finite-class and metric-entropy uniform estimates and use them to compare empirical finite-depth optimal values with continuous population optimal values.  These are fixed-class estimates: the admissible class is chosen independently of the sample, and the results are not presented as a full data-adaptive generalization theory for SGD-trained Transformers.
    \item We derive a Pontryagin first-order condition for the limiting population control problem.  The averaged Hamiltonian depends on the joint law of the state and adjoint, and the terminal adjoint contains the cross-entropy softmax residual $\softmax(H(X_T))-Y(X_0)$.
    \item We develop a conditional stability and design toolbox for the continuous-depth problem, including existence under compactness, local and global quadratic-growth implications, finite-depth stability of minimizers, continuous-to-discrete recovery, local SGD initialization under a PL condition, and a priori range estimates from compact input support and regularization.
\end{enumerate}

The rest of the paper is organized as follows.  \Cref{sec:transformer-formulation} fixes the data model, the token-level cross-entropy loss, and the simplified self-attention vector field used as a representative residual block.  \Cref{sec:assumptions} states the compactness, smoothness, and readout hypotheses and proves the basic Lipschitz estimate for the cross-entropy loss.  \Cref{sec:finite-depth-approximation} proves the Euler trajectory estimate, the empirical-to-population bounds, and the finite-class and metric-entropy uniform estimates.  \Cref{sec:continuous-depth-control} studies the limiting continuous-depth mean field control problem, including its transport formulation and Pontryagin condition.  \Cref{sec:conditional-stability-results} then collects conditional stability results, including existence under compactness, local and global quadratic-growth implications, continuous-to-discrete recovery, and the precise assumptions under which these tools apply.  \Cref{sec:initialization-range} then discusses data-dependent initialization and a priori range estimates from compact input support. Finally, \Cref{sec:discussion-conclusion} summarizes the scope of the theory, its limitations, and several directions for extension.

\section{Transformer formulation}
\label{sec:transformer-formulation}

\subsection{Data, states, and logits}

Fix a sequence length $L$, a hidden dimension $d_x$, and a vocabulary size $d_v$.  A hidden state is a matrix
\begin{align*}
    X=(x_0,\ldots,x_{L-1})^\top\in\R^{L\times d_x},
\end{align*}
where $x_l$ is the representation of the token at position $l\in\{0,\ldots,L-1\}$.  The vocabulary simplex is
\begin{align*}
    \DeltaV=\left\{p=(p_1,\ldots,p_{d_v})\in[0,1]^{d_v}:\sum_{j=1}^{d_v}p_j=1\right\}.
\end{align*}
Let $e_1,\ldots,e_{d_v}$ denote the canonical basis of $\R^{d_v}$.  The primitive data law is a probability measure $\mu_X\in\calP(\R^{L\times d_x})$ for the input sequence $X_0$.  For each $l\in\{0,\ldots,L-1\}$, the next-token label is a measurable map
\begin{align*}
    Y_l:\R^{L\times d_x}\to\{e_1,\ldots,e_{d_v}\}\subset\DeltaV\subset\R^{d_v}.
\end{align*}
Thus $Y_l(X_0)\in\R^{d_v}$ is the one-hot next-token label at position $l\in\{0,\ldots,L-1\}$, completely determined by $X_0$.  We write
\begin{align*}
    Y(X_0)=\bigl(Y_0(X_0),\ldots,Y_{L-1}(X_0)\bigr)^\top.
\end{align*}
 Sampling means drawing $X_0\sim\mu_X$ and then evaluating the deterministic labels $Y_l(X_0)$ in the terminal cross-entropy loss.  We assume throughout that the initial hidden states lie in a compact set or have been truncated to one; this assumption is made precise below.

The readout map $H:\R^{L\times d_x}\to\R^{L\times d_v}$ assigns token-level vocabulary logits.  In the simplest case,
\begin{align*}
    H(X)=XW_{\mathrm{out}},
\end{align*}
where $W_{\mathrm{out}}\in\R^{d_x\times d_v}$ is a fixed or separately trained output matrix.  The token-averaged cross-entropy loss is
\begin{equation}
\label{eq:ce-loss}
    \ell(Y,Z)
    =-\frac1L\sum_{l=0}^{L-1}\sum_{j=1}^{d_v}Y_{l,j}\log\softmax_j(Z_{l,:})
    =\frac1L\sum_{l=0}^{L-1}\left[\log\sum_{j=1}^{d_v}e^{Z_{l,j}}-Y_l^\top Z_{l,:}\right].
\end{equation}
Its gradient in $Z$ is
\begin{equation}
\label{eq:ce-gradient}
    \nabla_Z\ell(Y,Z)_{l,:}=\frac1L\bigl(\softmax(Z_{l,:})-Y_l\bigr),\qquad l=0,\ldots,L-1.
\end{equation}
The softmax residual in \eqref{eq:ce-gradient} is the main difference between the present cross-entropy setting and square-loss continuous-depth models.

\subsection{A Transformer vector field}

The analysis below applies to any vector field $F(X;\theta)$ satisfying the smoothness assumptions in \Cref{sec:assumptions}.  We now describe a typical masked self-attention example. The dimensions in the parameter tuple should be chosen so that every matrix product maps $\R^{L\times d_x}$ back to $\R^{L\times d_x}$.  For a single attention head with attention width $d_k$ and feed-forward width $d_{\rm ff}$, we take
\begin{align*}
    \theta&=(W_Q,W_K,W_V,W_O,W_1,b_1,W_2,b_2),\\
    W_Q,W_K,W_V&\in\R^{d_x\times d_k},\qquad
    W_O\in\R^{d_k\times d_x},\\
    W_1&\in\R^{d_x\times d_{\rm ff}},\qquad b_1\in\R^{d_{\rm ff}},\qquad
    W_2\in\R^{d_{\rm ff}\times d_x},\qquad b_2\in\R^{d_x}.
\end{align*}
The biases $b_1$ and $b_2$ are broadcast across the $L$ token positions. Define
\begin{align*}
    Q=XW_Q\in\R^{L\times d_k},
    \qquad K=XW_K\in\R^{L\times d_k},
    \qquad V=XW_V\in\R^{L\times d_k}.
\end{align*}
Let $\mathcal M\in(\{0,-\infty\})^{L\times L}$ be the causal mask, with $\mathcal M_{rs}=0$ for $0\le s\le r\le L-1$ and $\mathcal M_{rs}=-\infty$ for $0\le r<s\le L-1$.  The row-wise masked attention matrix is
\begin{equation}
\label{eq:attention-matrix}
    A_\theta(X)=\rowsoftmax\left(\frac{QK^\top}{\sqrt{d_k}}+\mathcal M\right).
\end{equation}
For the analysis, the notation with $-\infty$ is interpreted on the finite allowed coordinates.  Let $I_r=\{0,\ldots,r\}$.  Equivalently,
\begin{equation}
\label{eq:finite-row-masked-softmax}
    (A_\theta(X))_{rs}=\begin{cases}
    \displaystyle
    \frac{\exp(q_r^\top k_s/\sqrt{d_k})}
    {\sum_{u\in I_r}\exp(q_r^\top k_u/\sqrt{d_k})}, & s\in I_r,\\[1.2em]
    0, & s\notin I_r,
    \end{cases}
\end{equation}
where $q_r$ and $k_s$ are the rows of $Q$ and $K$.  Thus each active row is an ordinary smooth softmax on a finite-dimensional Euclidean space, and the forbidden coordinates are fixed zeros.  One may also replace $-\infty$ by a finite mask $-a$ and pass to the limit $a\to\infty$; the limiting expression relevant for differentiability is the finite-row formula above.
Thus $A_\theta(X)V W_O\in\R^{L\times d_x}$, and the feed-forward term also lies in $\R^{L\times d_x}$. A full pre-normalization Transformer block can be written schematically as
\begin{align}
\label{eq:full-transformer-block}
    U&=X+\operatorname{MHA}(\operatorname{LN}(X);\theta_{\rm att}),\notag\\
    \mathcal B(X;\theta)&=U+\operatorname{FFN}(\operatorname{LN}(U);\theta_{\rm ff}),
\end{align}
where, for $X\in\R^{L\times d_x}$, layer normalization is applied row-wise by
\begin{align}
\label{eq:row-layer-normalization}
    \operatorname{LN}(X)_{l,:}
    =\gamma\odot
    \frac{x_l-\bar x_l\mathbf 1_{d_x}}
    {\sqrt{d_x^{-1}\norm{x_l-\bar x_l\mathbf 1_{d_x}}^2+\epsilon_{\rm LN}}}
    +\beta,
    \qquad
    \bar x_l=d_x^{-1}\mathbf 1_{d_x}^\top x_l,
\end{align}
with fixed or trainable gain and bias vectors $\gamma,\beta\in\R^{d_x}$.  A multi-head attention map with $H_{\rm att}$ heads may be written as
\begin{align}
\label{eq:mha-formula}
    \operatorname{MHA}(X;\theta_{\rm att})
    =\operatorname{Concat}_{h=1}^{H_{\rm att}}
    \bigl(A_{\theta,h}(X)X W_{V,h}\bigr)W_O^{\rm mha},
\end{align}
where
\begin{align*}
    A_{\theta,h}(X)=\rowsoftmax\left(\frac{(XW_{Q,h})(XW_{K,h})^\top}{\sqrt{d_{k,h}}}+\mathcal M\right).
\end{align*}
Equivalently, one may absorb the concatenation and output projection into a finite sum of projected heads.  In either notation, $\operatorname{MHA}(X;\theta_{\rm att})\in\R^{L\times d_x}$.
The feed-forward part is
\begin{align*}
    \operatorname{FFN}(U;\theta_{\rm ff})=\sigma(UW_1+\mathbf 1_L b_1^\top)W_2+\mathbf 1_L b_2^\top\in \R^{L\times d_x}.
\end{align*}
It is applied row-wise, where $\sigma$ is a smooth activation applied componentwise such as the sigmoid activation function. The residual vector field associated with such a block is the increment $\mathcal B(X;\theta)-X$, possibly rescaled.
To simplify the main analysis, we use the following single-head smooth model vector field, obtained from \eqref{eq:full-transformer-block} by suppressing layer normalization, using one attention head, and combining the two residual increments into one map:
\begin{equation}
\label{eq:transformer-vector-field}
    F(X;\theta)=A_\theta(X)VW_O+\sigma(XW_1+\mathbf 1_Lb_1^\top)W_2+\mathbf 1_Lb_2^\top.
\end{equation}
This vector field is a deliberately simplified Transformer block.  It retains the two components that matter for the continuous-depth approximation--a masked self-attention term and a pointwise feed-forward term--and absorbs other architectural details into the abstract map $F$.  This simplification is legitimate for the present error analysis since the proofs use only uniform boundedness, Lipschitz continuity in the hidden state, and Lipschitz dependence on the layer parameter.  Multi-head attention can be represented by concatenating or summing finitely many attention maps inside $F$, while residual scaling, smooth layer-normalization approximations, and a trainable output projection merely change the constants in these local bounds.  Thus the formula \eqref{eq:transformer-vector-field} should be read as a model vector field rather than as an architectural restriction essential to the Euler or empirical-risk estimates. The important point is that the residual architecture has the form
\begin{align*}
    X_{k+1}=X_k+\eps\zeta(k\eps)F(X_k;\theta_k).
\end{align*}

\begin{proposition}[Checkable bounds for the simplified single-head block]
\label{prop:single-head-checkable}
Fix $R,B<\infty$ and define
\begin{align*}
    K_R&=\{X\in\R^{L\times d_x}:\norm{X}\le R\},\\
    \Theta_B&=\{\theta:\norm{W_Q},\norm{W_K},\norm{W_V},\norm{W_O},\norm{W_1},\norm{W_2},\norm{b_1},\norm{b_2}\le B\}.
\end{align*}
Assume that $\sigma\in C^2(\R)$ and set
\begin{align*}
    U_{R,B}:=RB+B,
    \qquad
    S_m(R,B):=\max_{|u|\le U_{R,B}} |\sigma^{(m)}(u)|,
    \qquad m=0,1,2.
\end{align*}
For the finite-row masked softmax definition \eqref{eq:finite-row-masked-softmax}, the vector field \eqref{eq:transformer-vector-field} satisfies, on $K_R\times\Theta_B$,
\begin{align*}
    \norm{F(X;\theta)}&\le B_F,\\
    \norm{F(X;\theta)-F(X';\theta)}&\le L_F\norm{X-X'},\\
    \norm{D_XF(X;\theta)-D_XF(X';\theta)}&\le L_F^{(1)}\norm{X-X'},\\
    \norm{F(X;\theta)-F(X;\vartheta)}&\le L_\theta\norm{\theta-\vartheta},
\end{align*}
for all $X,X'\in K_R$ and $\theta,\vartheta\in\Theta_B$.  The constants $B_F,L_F,L_F^{(1)}$, and $L_\theta$ depend only on
\begin{align*}
    L,\ d_x,\ d_k,\ d_{\rm ff},\ R,\ B,
    \quad\text{and}\quad S_0(R,B),S_1(R,B),S_2(R,B),
\end{align*}
and on the fixed mask pattern, but not on the particular $X,X',\theta$, or $\vartheta$.  If the readout $H$ is $C^1$ on $K_R$, then Assumption \ref{ass:readout} also holds on $K_R$ with constants depending on the corresponding suprema of $H$ and $D H$.
\end{proposition}

\begin{proof}
For $X\in K_R$ and $\theta\in\Theta_B$, all quantities entering the attention scores, values, output projection, feed-forward preactivations, and biases are contained in compact sets determined by $L,d_x,d_k,d_{\rm ff},R$, and $B$.  In particular, the allowed attention scores $q_r^\top k_s/\sqrt{d_k}$ are bounded by a constant depending only on these parameters.  On each row $I_r$, the map from the allowed scores to the probabilities in \eqref{eq:finite-row-masked-softmax} is the ordinary finite-dimensional softmax and is $C^\infty$ with bounded first and second derivatives on the relevant compact score set.  The forbidden coordinates are constant zeros and therefore contribute no derivatives.

The attention term $A_\theta(X)VW_O$ is a composition of bounded matrix multiplications, a finite-row softmax, and products with parameters in $\Theta_B$.  Hence the term and its first two derivatives with respect to $X$ are uniformly bounded on $K_R\times\Theta_B$, with constants depending only on the displayed dimensions and bounds.  The feed-forward term
\begin{align*}
    \sigma(XW_1+\mathbf 1_L b_1^\top)W_2+\mathbf 1_Lb_2^\top
\end{align*}
is bounded by the same type of constants since the preactivations lie in $[-U_{R,B},U_{R,B}]$ and the derivatives of $\sigma$ up to order two are bounded by $S_0,S_1,S_2$ on this interval.  Combining the two terms gives a uniform bound on $F$ and on the first two $X$-derivatives of $F$.  The mean-value theorem then gives the Lipschitz bound in $X$ and the Lipschitz bound for $D_XF$.

The same compactness argument applied to the first derivatives of the formula with respect to the entries of $W_Q,W_K,W_V,W_O,W_1,b_1,W_2,b_2$ gives a uniform bound on $D_\theta F$ on $K_R\times\Theta_B$.  Another application of the mean-value theorem gives the stated Lipschitz dependence on $\theta$.  The statement for $H$ follows from the extreme-value theorem on the compact set $K_R$.
\end{proof}

\subsection{Discrete-depth and continuous-depth objectives}

Let $T>0$ be fixed and let $M\in\N$ and $\eps>0$ satisfy $M\eps=T$.  Given a layer sequence $\theta_0,\ldots,\theta_{M-1}$, define the discrete Transformer trajectory by
\begin{equation}
\label{eq:discrete-transformer}
    X_{k+1}^{\eps}=X_k^{\eps}+\eps\zeta(t_k)F(X_k^{\eps};\theta_k),
    \qquad t_k=k\eps,
    \qquad X_0^{\eps}=X_0.
\end{equation}
We denote by $\thetaeps$ the piecewise constant interpolation
\begin{equation}
\label{eq:piecewise-control}
    \thetaeps_t=\theta_k,
    \qquad t\in[t_k,t_{k+1}).
\end{equation}
The associated continuous-depth trajectory is
\begin{equation}
\label{eq:continuous-transformer}
    \dot X_t=\zeta(t)F(X_t;\thetaeps_t),
    \qquad X_0=X_0.
\end{equation}
The population finite-depth and continuous-depth risks are
\begin{align}
\label{eq:population-discrete-risk}
    \mcJ_\eps(\thetaeps)
    &=\E_{X_0\sim\mu_X}\bigl[\ell(Y(X_0),H(X_M^{\eps}))\bigr],\\
\label{eq:population-continuous-risk}
    \mcJ(\thetaeps)
    &=\E_{X_0\sim\mu_X}\bigl[\ell(Y(X_0),H(X_T))\bigr].
\end{align}
For independent training inputs $X_0^1,\ldots,X_0^N$ with common law $\mu_X$, set $Y^i=Y(X_0^i)$.  The empirical finite-depth training loss is
\begin{equation}
\label{eq:empirical-risk}
    \mcJ_{N,\eps}(\thetaeps)
    =\frac1N\sum_{i=1}^N \ell(Y^i,H(X_{M}^{\eps,i})),
\end{equation}
where $X_k^{\eps,i}$ solves \eqref{eq:discrete-transformer} from $X_0^i$.

\section{Assumptions and preliminary estimates}
\label{sec:assumptions}

We use the Frobenius norm on matrices and the corresponding Euclidean norm after vectorization.

\begin{assumption}[Depth profile and admissible controls]
\label{ass:controls}
The depth profile $\zeta:[0,T]\to[0,\infty)$ is bounded and Lipschitz.  There exist constants $\zeta_+$ and $L_\zeta$ such that
\begin{align*}
    0\le \zeta(t)\le \zeta_+,
    \qquad \abs{\zeta(t)-\zeta(s)}\le L_\zeta\abs{t-s}.
\end{align*}
The control values belong to a compact set $\Theta\subset\R^{d_\theta}$, and admissible controls are measurable maps $\theta:[0,T]\to\Theta$.
\end{assumption}

\begin{assumption}[Regularity of the Transformer vector field]
\label{ass:field}
There are constants $B_F,L_F,L_F^{(1)}<\infty$ such that, on a relevant invariant compact set $K\subset\R^{L\times d_x}$, for all $\theta\in\Theta$,
\begin{align}
\label{eq:bounded-vector-field-assumptions}
    \norm{F(X;\theta)}&\le B_F,\\
    \norm{F(X;\theta)-F(X';\theta)}&\le L_F\norm{X-X'},\\
    \norm{D_XF(X;\theta)-D_XF(X';\theta)}&\le L_F^{(1)}\norm{X-X'}.
\end{align}
Moreover, the discrete and continuous trajectories considered below remain in $K$ on $[0,T]$.
\end{assumption}

\begin{remark}[On the compactness assumption]
The compactness condition can be interpreted in several standard ways.  One may project the dynamics onto a compact parameter/state domain, impose an invariant-domain condition, or work conditionally on an a priori bound for the trajectories; also see Proposition \ref{prop:uniform-local-compactness}.  For the explicit attention map \eqref{eq:transformer-vector-field}, smoothness and Lipschitz bounds hold on every bounded set since the masked softmax is smooth on the finite allowed rows and the remaining operations are polynomial or smooth.  Thus the constants in Assumption \ref{ass:field} are local-in-state constants. 
More explicitly, Proposition \ref{prop:single-head-checkable} verifies these local constants for the simplified single-head block on a bounded state set and compact parameter set.  Therefore Assumption \ref{ass:field} should be read as a checkable local hypothesis for a projected, regularized, normalized, or otherwise range-controlled model.  It is not a global assertion about an unconstrained Transformer parameterization or about all possible training trajectories.

Layer normalization gives another reason why compact invariant state sets are natural.  For a token vector $x\in\R^{d_x}$, let
\begin{align*}
    \bar x=\frac1{d_x}(\mathbf 1_{d_x}^\top x)\mathbf 1_{d_x},
    \qquad
    P_0x=x-\bar x,
    \qquad
    H_0=\{u\in\R^{d_x}:\mathbf 1_{d_x}^\top u=0\}.
\end{align*}
The idealized layer-normalized direction is
\begin{align*}
    \operatorname{LN}_0(x)
    =\frac{P_0x}{\sqrt{d_x^{-1}\norm{P_0x}^2+\epsilon_{\rm LN}}}
    \in H_0,
\end{align*}
so that
\begin{align*}
    \norm{\operatorname{LN}_0(x)}^2
    =\frac{\norm{P_0x}^2}{d_x^{-1}\norm{P_0x}^2+\epsilon_{\rm LN}}
    \le d_x.
\end{align*}
In the zero-stabilization idealization $\epsilon_{\rm LN}=0$ and $P_0x\ne0$, this gives $\norm{\operatorname{LN}_0(x)}=\sqrt{d_x}$; hence the normalized token lies on the sphere $\sqrt{d_x}S(H_0)$ inside the mean-zero subspace.  With fixed gain and bias vectors $\gamma,\beta\in\R^{d_x}$, the affine normalized token belongs to
\begin{align*}
    K_{\rm LN}^{(1)}
    =\{\gamma\odot u+\beta:u\in H_0,\ \norm{u}\le \sqrt{d_x}\},
\end{align*}
which is compact.  For a sequence of $L$ tokens, this yields the product compact set
\begin{align*}
    K_{\rm LN}=(K_{\rm LN}^{(1)})^L\subset\R^{L\times d_x}.
\end{align*}
Thus, the hidden state can be regarded as evolving on a product of spheres or closed balls in mean-zero subspaces.  This is the sense in which layer normalization can justify a compact state manifold; see \citet{GeshkovskiLetrouitPolyanskiyRigollet2025}, where layer normalization is used to project token dynamics to the sphere.
In the simplified vector field \eqref{eq:transformer-vector-field}, however, layer normalization has been suppressed.  Hence the compact set $K$ in Assumption \ref{ass:field} is an explicit modeling or localization assumption for that simplified equation.  If one studies the normalized block \eqref{eq:full-transformer-block} instead, the same type of local estimates apply after replacing $F$ by the normalized residual increment and using the bounded image of the stabilized normalization map.
\end{remark}

\begin{assumption}[Readout and bounded loss]
\label{ass:readout}
The readout map $H:K\to\R^{L\times d_v}$ is Lipschitz with constant $L_H$ and is bounded by $B_H$ on $K$.  Consequently, the cross-entropy loss is bounded on the relevant set: there exists $B_\ell<\infty$ such that
\begin{align*}
    0\le \ell(Y,H(X))\le B_\ell,
    \qquad Y\in\DeltaV^L,
    \quad X\in K.
\end{align*}
\end{assumption}

The next lemma converts the bounded-readout assumption into a quantitative continuity estimate for the cross-entropy loss.  This elementary bound is used repeatedly to pass from trajectory errors to risk errors.
\begin{lemma}[Lipschitz property of cross-entropy in logits]
\label{lem:ce-lip}
For all $Y\in\DeltaV^L$ and $Z,Z'\in\R^{L\times d_v}$,
\begin{align}
\label{eq:ce-lip}
    \abs{\ell(Y,Z)-\ell(Y,Z')}
    \le \frac{2}{\sqrt L}\norm{Z-Z'}.
\end{align}
In particular, under Assumption \ref{ass:readout},
\begin{align}
\label{eq:ce-h-lip}
    \abs{\ell(Y,H(X))-\ell(Y,H(X'))}
    \le L_\ell \norm{X-X'},
    \qquad L_\ell:=\frac{2L_H}{\sqrt L}.
\end{align}
\end{lemma}

\begin{proof}
By \eqref{eq:ce-gradient}, each row of $\nabla_Z\ell$ equals $L^{-1}(\softmax(Z_{l,:})-Y_l)$ for $l=0,\ldots,L-1$.  Since both $\softmax(Z_{l,:})$ and $Y_l$ are probability vectors, their Euclidean distance is at most $2$.  Therefore
\begin{align*}
    \norm{\nabla_Z\ell(Y,Z)}^2
    \le \sum_{l=0}^{L-1} \frac{4}{L^2}=\frac4L.
\end{align*}
The mean-value theorem gives \eqref{eq:ce-lip}.  Combining this with the Lipschitz property of $H$ gives \eqref{eq:ce-h-lip}.
\end{proof}

\section{Finite-depth approximation and statistical error estimates}
\label{sec:finite-depth-approximation}
This section gives the two basic quantitative approximations used throughout the paper.  First, the residual Transformer recursion is compared with the continuous controlled flow by an Euler-type pathwise estimate.  Second, the empirical finite-depth cross-entropy risk is compared with the population continuous-depth risk.  The section also records uniform versions over finite and entropy-controlled classes of layer sequences, which are later used to compare optimal values and approximate minimizers.
\subsection{Discrete-to-continuous trajectory approximation}

The first approximation is purely deterministic and concerns one input sequence.  The same estimate holds uniformly over all samples since the constants are independent of the initial condition as long as the trajectory remains in $K$.

\begin{theorem}[Pathwise Euler approximation]
\label{thm:pathwise}
Assume Assumptions \ref{ass:controls} and \ref{ass:field}.  Let $X_k^\eps$ solve the discrete Transformer recursion \eqref{eq:discrete-transformer}, and let $X_t$ solve the continuous controlled flow \eqref{eq:continuous-transformer} with the same initial condition and the piecewise constant control \eqref{eq:piecewise-control}.  Then there exists a constant $C_T<\infty$, depending only on $T$, $\zeta_+$, $L_\zeta$, and the constants in Assumption \ref{ass:field}, such that
\begin{equation}
\label{eq:pathwise-error}
    \max_{0\le k\le M}\norm{X_k^\eps-X_{t_k}}
    \le C_T\eps.
\end{equation}
Consequently,
\begin{equation}
\label{eq:pathwise-loss-error}
    \abs{\ell(Y,H(X_M^\eps))-\ell(Y,H(X_T))}
    \le L_\ell C_T\eps.
\end{equation}
\end{theorem}

\begin{proof}
Set
\begin{align*}
    \delta_k=X_k^\eps-X_{t_k},\qquad k=0,\ldots,M.
\end{align*}
For $t\in[t_k,t_{k+1})$, the continuous solution satisfies
\begin{align*}
    X_{t_{k+1}}
    =X_{t_k}+\int_{t_k}^{t_{k+1}}\zeta(s)F(X_s;\theta_k)\dd s.
\end{align*}
Subtracting this identity from \eqref{eq:discrete-transformer} gives
\begin{align*}
    \delta_{k+1}
    =\delta_k+\eps\zeta(t_k)\bigl[F(X_k^\eps;\theta_k)-F(X_{t_k};\theta_k)\bigr]-r_k,
\end{align*}
where
\begin{align*}
    r_k=\int_{t_k}^{t_{k+1}}\left[\zeta(s)F(X_s;\theta_k)-\zeta(t_k)F(X_{t_k};\theta_k)\right]\dd s.
\end{align*}
We now estimate the local truncation term.  Add and subtract $\zeta(t_k)F(X_s;\theta_k)$ inside the integrand.  By the Lipschitz property of $\zeta$ and the boundedness of $F$,
\begin{align*}
    \norm{r_k}
    &\le \int_{t_k}^{t_{k+1}}\abs{\zeta(s)-\zeta(t_k)}\norm{F(X_s;\theta_k)}\dd s
    +\int_{t_k}^{t_{k+1}}\zeta(t_k)\norm{F(X_s;\theta_k)-F(X_{t_k};\theta_k)}\dd s\\
    &\le \int_{t_k}^{t_{k+1}} L_\zeta(s-t_k)B_F\dd s
    +\int_{t_k}^{t_{k+1}} \zeta_+L_F\norm{X_s-X_{t_k}}\dd s.
\end{align*}
On $[t_k,t_{k+1})$, the control is fixed at $\theta_k$ and
\begin{align*}
    \norm{\dot X_u}=\norm{\zeta(u)F(X_u;\theta_k)}\le \zeta_+B_F.
\end{align*}
Hence, for $s\in[t_k,t_{k+1}]$,
\begin{align*}
    \norm{X_s-X_{t_k}}
    \le \int_{t_k}^s\norm{\dot X_u}\dd u
    \le \zeta_+B_F(s-t_k).
\end{align*}
Substituting this estimate into the preceding bound gives
\begin{align*}
    \norm{r_k}
    &\le \left(L_\zeta B_F+\zeta_+^2L_FB_F\right)
    \int_{t_k}^{t_{k+1}}(s-t_k)\dd s
    \le C\eps^2.
\end{align*}
Thus
\begin{align*}
    \norm{\delta_{k+1}}
    \le (1+\eps\zeta_+L_F)\norm{\delta_k}+C\eps^2.
\end{align*}
Since $\delta_0=0$, the discrete Gronwall inequality gives
\begin{align*}
    \max_{0\le k\le M}\norm{\delta_k}\le C_T\eps.
\end{align*}
The loss estimate follows from \Cref{lem:ce-lip} and Assumption \ref{ass:readout}.
\end{proof}

\begin{corollary}[Population discretization error]
\label{cor:population-discretization}
Under the assumptions of \Cref{thm:pathwise} and Assumption \ref{ass:readout},
\begin{equation}
\label{eq:pop-disc-error}
    \abs{\mcJ_\eps(\thetaeps)-\mcJ(\thetaeps)}
    \le C_T\eps.
\end{equation}
\end{corollary}

\begin{proof}
For each input $X_0$ sampled from $\mu_X$, the label sequence $Y(X_0)$ is fixed by the deterministic label map.  Applying \eqref{eq:pathwise-loss-error} with $Y=Y(X_0)$ gives
\begin{align*}    
    \abs{\ell(Y(X_0),H(X_M^\eps(X_0)))-\ell(Y(X_0),H(X_T(X_0)))}
    \le L_\ell C_T\eps.
\end{align*}
Integrating this pointwise estimate with respect to $\mu_X(\dd X_0)$ yields
\begin{align*}
    \abs{\mcJ_\eps(\thetaeps)-\mcJ(\thetaeps)}
    &\le \int_{\R^{L\times d_x}}
    \abs{\ell(Y(X_0),H(X_M^\eps(X_0)))-\ell(Y(X_0),H(X_T(X_0)))}\,\mu_X(\dd X_0)\\
    &\le C_T\eps.
\end{align*}
The constant $C_T$ has been enlarged to absorb the loss Lipschitz constant.
\end{proof}

\subsection{Empirical-to-population approximation}

We next compare the empirical finite-depth risk with its population counterpart.  This is a statistical sampling estimate and is independent of the Euler approximation.

\begin{theorem}[Fixed-control empirical error]
\label{thm:fixed-control}
Assume Assumptions \ref{ass:controls}, \ref{ass:field} and \ref{ass:readout}.  Fix a layer sequence $\theta_0,\ldots,\theta_{M-1}$ and its interpolation $\thetaeps$.  Let $X_0^1,\ldots,X_0^N$ be independent samples from $\mu_X$, and set $Y^i=Y(X_0^i)$ for each $i$.  Then, for every $\delta\in(0,1)$, with probability at least $1-\delta$,
\begin{equation}
\label{eq:fixed-control-error}
    \abs{\mcJ_{N,\eps}(\thetaeps)-\mcJ(\thetaeps)}
    \le C_T\eps+B_\ell\sqrt{\frac{\log(2/\delta)}{2N}}.
\end{equation}
\end{theorem}

\begin{proof}
By the triangle inequality,
\begin{align*}
    \abs{\mcJ_{N,\eps}(\thetaeps)-\mcJ(\thetaeps)}
    \le \abs{\mcJ_{N,\eps}(\thetaeps)-\mcJ_\eps(\thetaeps)}
    +\abs{\mcJ_\eps(\thetaeps)-\mcJ(\thetaeps)}.
\end{align*}
The second term is bounded by \Cref{cor:population-discretization}.  For the first term, the random variables
\begin{align*}
    Z_i=\ell(Y^i,H(X_M^{\eps,i}))
\end{align*}
are independent and take values in $[0,B_\ell]$.  Hoeffding's inequality gives
\begin{align*}
    \Prob\left(\abs{\frac1N\sum_{i=1}^N Z_i-\E Z_i}>r\right)
    \le 2\exp\left(-\frac{2Nr^2}{B_\ell^2}\right).
\end{align*}
Choose
\begin{align*}
    r=B_\ell\sqrt{\frac{\log(2/\delta)}{2N}}.
\end{align*}
Then $2Nr^2/B_\ell^2=\log(2/\delta)$, and hence
\begin{align*}
    2\exp\left(-\frac{2Nr^2}{B_\ell^2}\right)
    =2\exp\bigl(-\log(2/\delta)\bigr)
    =\delta.
\end{align*}
Since
\begin{align*}
    \frac1N\sum_{i=1}^N Z_i=\mcJ_{N,\eps}(\thetaeps),
    \qquad
    \E Z_i=\mcJ_\eps(\thetaeps),
\end{align*}
Hoeffding's inequality gives, with probability at least $1-\delta$,
\begin{align*}
    \abs{\mcJ_{N,\eps}(\thetaeps)-\mcJ_\eps(\thetaeps)}
    \le B_\ell\sqrt{\frac{\log(2/\delta)}{2N}}.
\end{align*}
This is the standard bounded-difference concentration estimate; see, for example, \citet{BoucheronLugosiMassart2013}.
Combining this sampling estimate with the deterministic bound in \Cref{cor:population-discretization} proves \eqref{eq:fixed-control-error}.
\end{proof}

The fixed-control estimate is the statistically clean statement: the control is chosen before seeing the sample.  The uniform results below remain valid, but only for a fixed admissible class specified independently of the data.  They should therefore be read as oracle or class-comparison estimates, not as an explanation of why an SGD-trained Transformer selected adaptively from the same data generalizes. In particular, a layerwise grid with $q$ possible layer parameters has size $q^{T/\eps}$, giving a statistical term of order $\sqrt{T\log q/(N\eps)}$ (see Remark \ref{rem:size-control-class}), and the corresponding bound worsens as the continuous-depth discretization is refined.  A stronger generalization claim would require either a sharper complexity measure, such as Rademacher or PAC-Bayesian complexity, an algorithmic-stability analysis of the training procedure, or a temporally regular or low-dimensional control class whose complexity does not grow exponentially in $T/\eps$.

We now move from a single fixed control to a prescribed family of admissible layer sequences.  The next definition fixes the finite class over which the union-bound estimate is taken.
\begin{definition}[Finite-depth admissible control class]
\label{def:finite-control-class}
Let $M=T/\eps$ be the number of layers.  A finite-depth admissible control class is a subset
\begin{align*}
    \calA_\eps\subset\Theta^M
    =\{(\theta_0,\ldots,\theta_{M-1}):\theta_k\in\Theta\text{ for every }k\},
\end{align*}
whose elements are layer sequences $\thetaeps=(\theta_0,\ldots,\theta_{M-1})$.  Each layer sequence is identified with the piecewise constant control
\begin{align*}
    \thetaeps_t=\theta_k,
    \qquad t\in[t_k,t_{k+1}),
    \qquad t_k=k\eps.
\end{align*}
In the finite-class estimate \eqref{eq:finite-class-error} below we assume, in addition, that $\abs{\calA_\eps}<\infty$.
\end{definition}

\begin{theorem}[Uniform estimate over a finite control class]
\label{thm:finite-class}
Under the assumptions of \Cref{thm:fixed-control} and with $\calA_\eps$ as in Definition \ref{def:finite-control-class}, with probability at least $1-\delta$,
\begin{equation}
\label{eq:finite-class-error}
    \sup_{\thetaeps\in\calA_\eps}
    \abs{\mcJ_{N,\eps}(\thetaeps)-\mcJ(\thetaeps)}
    \le C_T\eps+B_\ell\sqrt{\frac{\log(2\abs{\calA_\eps}/\delta)}{2N}}.
\end{equation}
\end{theorem}

\begin{proof}
Fix $\thetaeps\in\calA_\eps$.  By \Cref{thm:fixed-control}, applied with confidence level $\delta/\abs{\calA_\eps}$,
\begin{align*}    
    \Prob\left(
    \abs{\mcJ_{N,\eps}(\thetaeps)-\mcJ(\thetaeps)}
    >C_T\eps+B_\ell\sqrt{\frac{\log(2\abs{\calA_\eps}/\delta)}{2N}}
    \right)
    \le \frac{\delta}{\abs{\calA_\eps}}.
\end{align*}
Taking the union bound over all controls in the finite class gives
\begin{align*}
    \Prob\left(
    \sup_{\thetaeps\in\calA_\eps}
    \abs{\mcJ_{N,\eps}(\thetaeps)-\mcJ(\thetaeps)}
    >C_T\eps+B_\ell\sqrt{\frac{\log(2\abs{\calA_\eps}/\delta)}{2N}}
    \right)
    \le \sum_{\thetaeps\in\calA_\eps}\frac{\delta}{\abs{\calA_\eps}}
    =\delta.
\end{align*}
Equivalently, the desired uniform estimate holds with probability at least $1-\delta$.
\end{proof}

\begin{remark}[Size of a discretized control class]\label{rem:size-control-class}
The quantity $\abs{\calA_\eps}$ measures the number of admissible layer sequences over which the empirical risk is optimized.  For example, if each layer parameter is chosen from a finite grid $\Theta_q\subset\Theta$ with $q$ grid points and the network has $M=T/\eps$ layers, then the unrestricted class satisfies
\begin{align*}
    \abs{\calA_\eps}\le q^M=q^{T/\eps}.
\end{align*}
Substituting this bound into \eqref{eq:finite-class-error} gives a statistical term of order
\begin{align*}
    B_\ell\sqrt{\frac{M\log q+\log(2/\delta)}{2N}}
    =B_\ell\sqrt{\frac{(T/\eps)\log q+\log(2/\delta)}{2N}}.
\end{align*}
Thus a naive layerwise grid can be statistically expensive when $\eps$ is small.  Smaller structured classes, shared parameters, or regularized controls reduce $\abs{\calA_\eps}$ and lead to sharper uniform estimates.
This scaling also limits the interpretation of the bound: if $\eps\downarrow0$ while $N$ is fixed or grows slowly, the grid-complexity term can dominate the Euler improvement.  Hence this finite-class estimate is useful for controlled comparisons over deliberately restricted classes, but it is not by itself a practical generalization bound for full layerwise Transformer training.
\end{remark}

\subsection{Entropy-controlled uniform estimates}

For infinite control classes one needs a complexity measure.  We state a simple covering-number version.  For two piecewise constant controls $\thetaeps$ and $\vartheta^\eps$, define
\begin{align*}
    d_\infty(\thetaeps,\vartheta^\eps)
    =\max_{0\le k\le M-1}\norm{\theta_k-\vartheta_k}.
\end{align*}

\begin{assumption}[Lipschitz dependence on the control]
\label{ass:theta-lip}
There exists $L_\theta<\infty$ such that, for all $X\in K$ and $\theta,\vartheta\in\Theta$,
\begin{align*}
    \norm{F(X;\theta)-F(X;\vartheta)}\le L_\theta\norm{\theta-\vartheta}.
\end{align*}
\end{assumption}

\begin{lemma}[Lipschitz dependence of the risk on layer parameters]
\label{lem:risk-lip-control}
Under Assumptions \ref{ass:controls}, \ref{ass:field}, \ref{ass:readout}, and \ref{ass:theta-lip}, there exists $L_{\calA,T}<\infty$ such that, for any two layer sequences $\thetaeps$ and $\vartheta^\eps$,
\begin{align}
\label{eq:risk-lip-control}
    \abs{\mcJ_{N,\eps}(\thetaeps)-\mcJ_{N,\eps}(\vartheta^\eps)}
    +\abs{\mcJ(\thetaeps)-\mcJ(\vartheta^\eps)}
    \le L_{\calA,T}d_\infty(\thetaeps,\vartheta^\eps).
\end{align}
\end{lemma}

\begin{proof}
Let $X_k^\eps$ and $\widetilde X_k^\eps$ be the discrete trajectories driven by $\thetaeps$ and $\vartheta^\eps$ from the same initial condition.  Applying the discrete recursion \eqref{eq:discrete-transformer} to the two layer sequences and using Assumptions \ref{ass:field} and \ref{ass:theta-lip} gives 
\begin{align}
\label{eq:two-control-recursion}
    \norm{X_{k+1}^\eps-\widetilde X_{k+1}^\eps}
    \le (1+\eps\zeta_+L_F)\norm{X_k^\eps-\widetilde X_k^\eps}
    +\eps\zeta_+L_\theta d_\infty(\thetaeps,\vartheta^\eps).
\end{align}

By the discrete Gronwall inequality,
\begin{align*}
    \max_{k\le M}\norm{X_k^\eps-\widetilde X_k^\eps}
    \le C_T d_\infty(\thetaeps,\vartheta^\eps).
\end{align*}
The same estimate for the continuous flows follows by the usual Gronwall inequality.  Therefore, for every input $X_0$ and deterministic label $Y(X_0)$,
\begin{align*}
&\abs{\ell(Y(X_0),H(X_M^{\eps,\thetaeps}))
      -\ell(Y(X_0),H(X_M^{\eps,\vartheta^\eps}))}  \\
&\qquad\le L_\ell
\norm{X_M^{\eps,\thetaeps}-X_M^{\eps,\vartheta^\eps}}
\le C_T d_\infty(\thetaeps,\vartheta^\eps),
\end{align*}
and the same bound holds for the continuous terminal states $X_T^{\thetaeps}$ and $X_T^{\vartheta^\eps}$.  Averaging the discrete bound over the empirical measure gives the estimate for $\mcJ_{N,\eps}$, while averaging the continuous bound with respect to $\mu_X$ gives the estimate for $\mcJ$.  Enlarging the constant yields \eqref{eq:risk-lip-control}.
\end{proof}

\begin{theorem}[Metric-entropy uniform estimate]
\label{thm:metric-entropy}
Let $\calA_\eps$ be a possibly infinite class of admissible layer sequences, and let $\mathcal C_r$ be an $r$-net of $\calA_\eps$ under $d_\infty$ with cardinality $\mathcal N(r,\calA_\eps,d_\infty)$.  Under Assumptions \ref{ass:controls}, \ref{ass:field}, \ref{ass:readout} and \ref{ass:theta-lip}, for every $r>0$ and $\delta\in(0,1)$, with probability at least $1-\delta$,
\begin{align}
\label{eq:metric-entropy-error}
    \sup_{\thetaeps\in\calA_\eps}\abs{\mcJ_{N,\eps}(\thetaeps)-\mcJ(\thetaeps)}
    \le C_T\eps+2L_{\calA,T}r
    +B_\ell\sqrt{\frac{\log(2\mathcal N(r,\calA_\eps,d_\infty)/\delta)}{2N}}.
\end{align}
\end{theorem}

\begin{proof}
By \Cref{thm:finite-class}, with probability at least $1-\delta$,
\begin{align*}
    \sup_{\vartheta^\eps\in\mathcal C_r}
    \abs{\mcJ_{N,\eps}(\vartheta^\eps)-\mcJ(\vartheta^\eps)}
    \le C_T\eps+B_\ell\sqrt{\frac{\log(2\abs{\mathcal C_r}/\delta)}{2N}}.
\end{align*}
For any $\thetaeps\in\calA_\eps$, choose $\vartheta^\eps\in\mathcal C_r$ with $d_\infty(\thetaeps,\vartheta^\eps)\le r$.  Then \Cref{lem:risk-lip-control} bounds the two approximation terms by $2L_{\calA,T}r$.  Taking the supremum proves the theorem.
\end{proof}

\begin{remark}[Estimating the covering number]
The metric $d_\infty$ treats a depth-$M$ control as a point in the product space $\Theta^M$.  If $\Theta\subset\R^{d_\theta}$ is contained in a Euclidean ball of radius $R$, then a standard volumetric argument gives
\begin{align*}
    \mathcal N(r,\Theta^M,d_\infty)
    \le \left(1+\frac{2R}{r}\right)^{d_\theta M}
    =\left(1+\frac{2R}{r}\right)^{d_\theta T/\eps}.
\end{align*}
Substituting this volumetric estimate into \eqref{eq:metric-entropy-error} gives the worst-case bound
\begin{align}
\label{eq:optimized-metric-entropy-bound}
    \sup_{\thetaeps\in\calA_\eps}
    \abs{\mcJ_{N,\eps}(\thetaeps)-\mcJ(\thetaeps)}
    \lesssim
    \eps+r+
    \sqrt{\frac{d_\theta T\log(1+2R/r)}{N\eps}}
    +\sqrt{\frac{\log(1/\delta)}{N}},
\end{align}
up to constants depending on $B_\ell,L_{\calA,T}$ and $C_T$.  Balancing the deterministic depth error, the net radius, and the entropy term by taking $r\asymp\eps$ and optimizing in $\eps$ gives
\begin{align*}
    \eps_\star\asymp r_\star
    \asymp \left(\frac{d_\theta T\log N}{N}\right)^{1/3},
\end{align*}
and therefore the optimized worst-case rate
\begin{align*}
    \sup_{\thetaeps\in\calA_{\eps_\star}}
    \abs{\mcJ_{N,\eps_\star}(\thetaeps)-\mcJ(\thetaeps)}
    \lesssim
    \left(\frac{d_\theta T\log N}{N}\right)^{1/3}
    +\sqrt{\frac{\log(1/\delta)}{N}}.
\end{align*}
For the net itself, the same calculation without the covering-radius term $2L_{\calA,T}r$ gives the same $N^{-1/3}$ balance after choosing $r$ of the same order as the final approximation radius.  If the admissible class imposes temporal regularity, parameter sharing, low-dimensional parameterization, or a bounded-variation constraint in depth, the covering number can be substantially smaller.  The estimate \eqref{eq:metric-entropy-error} is stated in terms of $\mathcal N(r,\calA_\eps,d_\infty)$ precisely to allow such structural improvements.
The sup metric over the full product $\Theta^M$ is intentionally conservative.  It ignores algorithmic bias, data-dependent localization, implicit regularization, and the fact that practical training explores a small subset of parameter space.  A sharper theory for trained Transformers would need a complexity notion adapted to the learned class or the training algorithm, for example localized Rademacher complexity, PAC-Bayesian bounds, compression, or algorithmic stability.
\end{remark}

\subsection{Comparison of minimizers and approximate training}

The uniform estimate \eqref{eq:finite-class-error} implies convergence of optimal values over controlled classes.  Let $\calA_\eps$ be a set of finite-depth controls and define
\begin{align*}
    V_{N,\eps}=\inf_{\thetaeps\in\calA_\eps}\mcJ_{N,\eps}(\thetaeps),
    \qquad
    V_{\eps}=\inf_{\thetaeps\in\calA_\eps}\mcJ(\thetaeps).
\end{align*}

\begin{theorem}[Optimal value comparison over a finite class]
\label{cor:value-comparison}
Under the assumptions of \Cref{thm:finite-class}, with probability at least $1-\delta$,
\begin{align}
\label{eq:value-comparison}
    \abs{V_{N,\eps}-V_\eps}
    \le C_T\eps+B_\ell\sqrt{\frac{\log(2\abs{\calA_\eps}/\delta)}{2N}}.
\end{align}
Moreover, if $\widehat\thetaeps$ is $\eta$-optimal for the empirical problem, namely
\begin{align*}
    \mcJ_{N,\eps}(\widehat\thetaeps)\le V_{N,\eps}+\eta,
\end{align*}
then
\begin{align}
\label{eq:approx-minimizer}
    \mcJ(\widehat\thetaeps)
    \le V_\eps+\eta+2\left(C_T\eps+B_\ell\sqrt{\frac{\log(2\abs{\calA_\eps}/\delta)}{2N}}\right).
\end{align}
\end{theorem}

\begin{proof}
The first claim follows from
\begin{align*}
    \abs{\inf_{\thetaeps}\mcJ_{N,\eps}(\thetaeps)-\inf_{\thetaeps}\mcJ(\thetaeps)}
    \le \sup_{\thetaeps}\abs{\mcJ_{N,\eps}(\thetaeps)-\mcJ(\thetaeps)}
\end{align*}
and \Cref{thm:finite-class}.  For the second claim, on the same event,
\begin{align*}
    \mcJ(\widehat\thetaeps)
    &\le \mcJ_{N,\eps}(\widehat\thetaeps)+\Gamma_{N,\eps}
    \le V_{N,\eps}+\eta+\Gamma_{N,\eps}
    \le V_\eps+\eta+2\Gamma_{N,\eps},
\end{align*}
where $\Gamma_{N,\eps}$ denotes the right-hand side of \eqref{eq:value-comparison}.
\end{proof}

The preceding theorem compares empirical and population risks within a fixed finite-depth class.  We next relate such finite-depth classes to the full continuous-time admissible class through a deterministic density assumption.
\begin{theorem}[Approximation of the full continuous control problem]
\label{thm:full-control-approximation}
Let $\calU$ be a continuous-time admissible control class and let
\begin{align*}
    V=\inf_{\theta\in\calU}\mcJ(\theta).
\end{align*}
Assume that each element of $\calA_\eps$ is identified with a piecewise constant control in $\calU$, so that $V\le V_\eps$.  Suppose moreover that there exists a deterministic approximation error $\eta_\eps\downarrow0$ such that, for every $\theta\in\calU$ and every $\eta>0$, there is $\vartheta^\eps\in\calA_\eps$ satisfying
\begin{align}
\label{eq:control-density}
    \mcJ(\vartheta^\eps)\le \mcJ(\theta)+\eta_\eps+\eta.
\end{align}
Then
\begin{align}
\label{eq:continuous-value-approx}
    0\le V_\eps-V\le \eta_\eps.
\end{align}
If, in addition, $\calA_\eps$ is finite, then with probability at least $1-\delta$,
\begin{align}
\label{eq:empirical-to-full-value}
    \abs{V_{N,\eps}-V}
    \le \eta_\eps+C_T\eps+B_\ell\sqrt{\frac{\log(2\abs{\calA_\eps}/\delta)}{2N}}.
\end{align}
\end{theorem}

\begin{proof}
Since $\calA_\eps\subset\calU$, we have $V\le V_\eps$.  To prove the reverse inequality up to $\eta_\eps$, fix $\eta>0$ and choose $\theta^\eta\in\calU$ such that
\begin{align*}
    \mcJ(\theta^\eta)\le V+\eta.
\end{align*}
By the approximation assumption \eqref{eq:control-density}, there exists $\vartheta^\eps\in\calA_\eps$ such that
\begin{align*}
    \mcJ(\vartheta^\eps)
    \le \mcJ(\theta^\eta)+\eta_\eps+\eta
    \le V+\eta_\eps+2\eta.
\end{align*}
Taking the infimum over $\calA_\eps$ gives $V_\eps\le V+\eta_\eps+2\eta$.  Letting $\eta\downarrow0$ proves \eqref{eq:continuous-value-approx}.  Finally, \Cref{cor:value-comparison} gives
\begin{align*}
    \abs{V_{N,\eps}-V}
    \le \abs{V_{N,\eps}-V_\eps}+\abs{V_\eps-V}
    \le C_T\eps+B_\ell\sqrt{\frac{\log(2\abs{\calA_\eps}/\delta)}{2N}}+\eta_\eps,
\end{align*}
which proves \eqref{eq:empirical-to-full-value}.
\end{proof}

\begin{remark}[On the density term]
The quantity $\eta_\eps$ is independent of the Euler error in \Cref{thm:pathwise} and of the sampling error in \Cref{thm:fixed-control}.  It is an approximation error of the admissible control class: it measures how well the finite-depth class $\calA_\eps$ can reproduce the best objective value available in the continuous-time class $\calU$.  Thus three different mechanisms appear in \eqref{eq:empirical-to-full-value}: $C_T\eps$ comes from replacing the continuous state equation by an Euler residual network, the $N^{-1/2}$ term comes from empirical sampling, and $\eta_\eps$ comes from restricting the controls to a finite-depth or finite-dimensional subclass.

For example, suppose that $\calU$ consists of Lipschitz controls $\theta:[0,T]\to\Theta$, with Lipschitz constant bounded by $L_{\calU}$, and that $F$ is Lipschitz in the control variable.  If $\calA_\eps$ contains the grid-sampled controls $\thetaeps_t=\theta(t_k)$ for $t\in[t_k,t_{k+1})$, then
\begin{align*}
    \int_0^T \norm{\theta_t-\thetaeps_t}\dd t\le C\eps.
\end{align*}
Let $X_t$ and $\widetilde X_t$ be the continuous flows driven by $\theta$ and by the sampled control $\thetaeps$, respectively, from the same initial input.  By Assumptions \ref{ass:field} and \ref{ass:theta-lip},
\begin{align*}
    \frac{\dd}{\dd t}\norm{X_t-\widetilde X_t}
    \le \zeta_+L_F\norm{X_t-\widetilde X_t}
    +\zeta_+L_\theta\norm{\theta_t-\thetaeps_t}.
\end{align*}
Gronwall's inequality gives
\begin{align*}
    \sup_{0\le t\le T}\norm{X_t-\widetilde X_t}
    \le C_T\int_0^T\norm{\theta_t-\thetaeps_t}\dd t
    \le C_T\eps.
\end{align*}
Applying the loss Lipschitz estimate \eqref{eq:ce-h-lip} at the terminal time and integrating over $\mu_X$ yields
\begin{align*}
    \abs{\mcJ(\thetaeps)-\mcJ(\theta)}\le C_T\eps.
\end{align*}
In this case one may take $\eta_\eps=O(\eps)$.  By contrast, if $\calU$ contains arbitrary measurable controls, pointwise grid sampling is not stable and no deterministic rate for $\eta_\eps$ follows without additional compactness, temporal regularity, relaxation, or a prescribed approximation scheme.
\end{remark}

\section{Continuous-depth mean field control problem}
\label{sec:continuous-depth-control}
This section formulates the limiting population control problem.  The first part defines the transport equation and the Lagrangian objective, and the second derives the Pontryagin condition. 

\subsection{Mean-field transport formulation and Lagrangian objective}

The continuous-depth population problem is described by the law of the hidden state alone.  Let $\Phi_t^\theta$ be the flow map generated by
\begin{align}\label{def:flowmap}  
    \dot X_t=\zeta(t)F(X_t;\theta_t),\qquad X_0=x,
\end{align}
and define
\begin{align}\label{def:density-t}  
    \rho_t^\theta=\calL(X_t^\theta)=(\Phi_t^\theta)_\#\mu_X\in\calP(\R^{L\times d_x}).
\end{align}
Then $\rho_t^\theta$ solves the first-order transport equation
\begin{equation}
\label{eq:transport}   \partial_t\rho_t^\theta+\nabla_X\cdot\bigl(\zeta(t)F(X;\theta_t)\rho_t^\theta\bigr)=0,
    \qquad \rho_0^\theta=\mu_X.
\end{equation}
The labels are not part of the transported state; they are evaluated from the initial input through the deterministic map $X_0\mapsto Y(X_0)$.  Therefore the limiting first-order mean field control problem is written in Lagrangian form as
\begin{equation}
\label{eq:mf-control}
    \inf_{\theta\in\calU}\mcJ(\theta)
    =\inf_{\theta\in\calU}\int_{\R^{L\times d_x}} \ell\bigl(Y(X_0),H(\Phi_T^\theta(X_0))\bigr)\,\mu_X(\dd X_0),
\end{equation}
where $\calU$ may be $L^\infty([0,T];\Theta)$ when one only studies admissible measurable controls, or a compact subclass of it when existence of minimizers is needed, such as an equicontinuous family of controls, a uniformly bounded Lipschitz class, or a bounded-variation admissible class with a topology compatible with compactness; also see \Cref{thm:existence-global-minimizer} and Remark \ref{rem:compact}.  The word ``mean field'' refers to the fact that the state of the control problem is the population law $\rho_t^\theta$ of hidden representations, not a single trajectory.  Since the control is common to all samples, this is a deterministic first-order mean field type control problem rather than a mean field game.

\subsection{Pontryagin first-order condition}

We state the Pontryagin condition in a precise local model.  Let $\Theta\subset\R^{d_\theta}$ be compact and convex, and let the admissible controls be a subset $\calU\subset L^\infty([0,T];\Theta)$.  Local minimality is understood with respect to a chosen topology on $\calU$, for instance the strong $L^2([0,T];\R^{d_\theta})$ topology on a local Hilbert chart or the strong $L^1$ topology when only measurable controls are used.  Since $\Theta$ may have a boundary, the generic necessary condition is a variational inequality; the interior stationarity equation is only the special case in which $\theta_t^\ast$ lies in the interior of $\Theta$ for almost every $t$.

\begin{assumption}[Differentiability for the Pontryagin condition]
\label{ass:pmp-differentiability}
Let $K_{\rm in}:=\operatorname{supp}\mu_X$ be the input support and
\begin{align*}
    K_{\calU}:=\overline{\{\Phi_t^\theta(x):x\in K_{\rm in},\ t\in[0,T],\ \theta\in\calU\}}\subset\R^{L\times d_x}
\end{align*}
be the compact reachable state set generated by the admissible controls under consideration.  The vector field $F$ is continuously differentiable in both $X$ and $\theta$ on $K_{\calU}\times\Theta$, the readout $H$ is continuously differentiable on $K_{\calU}$, and the admissible parameter set $\Theta$ is convex.
Moreover, $D_XF$, $D_\theta F$, $DH$, and the terminal loss gradient along $H(K_{\calU})$ are uniformly bounded.  These uniform bounds, together with compactness of $K_{\calU}$ and $\Theta$, are the hypotheses used below to justify first variations and differentiation under the expectation.
\end{assumption}

For a data input $X_0$ with associated deterministic label $Y(X_0)$, define the state equation
\begin{align}
\label{eq:pmp-state}
    \dot X_t=\zeta(t)F(X_t;\theta_t),
    \qquad X_0=X_0.
\end{align}
The adjoint equation is
\begin{align}
\label{eq:pmp-adjoint}
    -\dot P_t=\zeta(t)D_XF(X_t;\theta_t)^\top P_t,
    \qquad
    P_T=D H(X_T)^\top\nabla_Z\ell(Y(X_0),H(X_T)).
\end{align}
Using \eqref{eq:ce-gradient}, the terminal adjoint is explicitly
\begin{align}
\label{eq:pmp-terminal}
    P_T=D H(X_T)^\top\left[\frac1L\bigl(\softmax(H(X_T))-Y(X_0)\bigr)\right],
\end{align}
where the softmax is applied row-wise.

Define the Hamiltonian
\begin{align}
\label{eq:pmp-hamiltonian}
    \mathcal H(t,X,P,\theta)
    :=\zeta(t)\inner{P}{F(X;\theta)}.
\end{align}
For a candidate optimizer $\theta^\ast$, define the averaged Hamiltonian along the optimal state-adjoint characteristics by
\begin{align}
\label{eq:averaged-hamiltonian}
    \overline{\mathcal H}_{\theta^\ast}(t,\theta)
    :=\E\bigl[\mathcal H(t,X_t^\ast,P_t^\ast,\theta)\bigr].
\end{align}
The expectation is over $X_0\sim\mu_X$; equivalently, $X_t^\ast=X_t^\ast(X_0)$ and $P_t^\ast=P_t^\ast(X_0)$ are random variables induced by the characteristic starting from $X_0$.

\begin{proposition}[Pontryagin variational inequality]
\label{prop:pmp}
Let $\theta^\ast\in\calU$ be a local minimizer of \eqref{eq:mf-control} in the chosen local topology, and suppose that Assumption \ref{ass:pmp-differentiability} holds.  Let $(X_t^\ast,P_t^\ast)$ solve \eqref{eq:pmp-state}--\eqref{eq:pmp-adjoint} under $\theta^\ast$.  Then, for every admissible comparison control $\vartheta\in\calU$\footnote{Here $\vartheta=(\vartheta_t)_{t\in[0,T]}$ is an admissible comparison control, and $\vartheta_t$ denotes its value at depth time $t$.} such that
\begin{align*}
    \theta^\alpha:=\theta^\ast+\alpha(\vartheta-\theta^\ast)\in\calU
\end{align*}
for all sufficiently small $\alpha\in[0,1]$, one has
\begin{align}
\label{eq:pmp-vi}
    \int_0^T
    \left\langle\nabla_\theta\overline{\mathcal H}_{\theta^\ast}(t,\theta_t^\ast),
    \vartheta_t-\theta_t^\ast\right\rangle\dd t\ge 0.
\end{align}
Equivalently,
\begin{align}
\label{eq:pmp-vi-expanded}
    \int_0^T \zeta(t)\,\E\left[
    \inner{P_t^\ast}{D_\theta F(X_t^\ast;\theta_t^\ast)(\vartheta_t-\theta_t^\ast)}
    \right]\dd t\ge 0.
\end{align}
If $\calU=L^\infty([0,T];\Theta)$ and needle variations with arbitrary values in $\Theta$ are admissible, then the variational inequality localizes to the averaged Hamiltonian condition
\begin{align}
\label{eq:pmp-averaged-hamiltonian-min}
    \overline{\mathcal H}_{\theta^\ast}(t,\theta_t^\ast)
    \le \overline{\mathcal H}_{\theta^\ast}(t,\theta)
    \qquad\text{for every }\theta\in\Theta,\quad\text{for a.e. }t\in[0,T].
\end{align}
If, in addition, $\theta^\ast_t$ is an interior point of $\Theta$ for almost every $t$, then
\begin{align}
\label{eq:pmp-stationarity}
    \nabla_\theta\overline{\mathcal H}_{\theta^\ast}(t,\theta_t^\ast)
    =\zeta(t)\E\left[D_\theta F(X_t^\ast;\theta_t^\ast)^\top P_t^\ast\right]=0
    \qquad \text{for a.e. }t\in[0,T]\text{ with }\zeta(t)>0.
\end{align}
Equivalently, the interior stationarity condition can be written as
\begin{align}
\label{eq:pmp-law-stationarity}
    \int D_\theta F(X_t^\ast(x);\theta_t^\ast)^\top P_t^\ast(x)\,\mu_X(\dd x)=0.
\end{align}
\end{proposition}

\begin{proof}
Let $\theta^\alpha=\theta^\ast+\alpha(\vartheta-\theta^\ast)$ and let $X_t^\alpha$ be the corresponding state.  Since $F$, $D_XF$, and $D_\theta F$ are uniformly bounded on the compact set $K_{\calU}\times\Theta$, the standard continuous-dependence and variational equations for ODEs imply that the directional derivative
\begin{align*}
    \dot X_t^0=\left.\frac{d}{d\alpha}X_t^\alpha\right|_{\alpha=0}
\end{align*}
exists uniformly for $X_0\in K_{\rm in}$ and solves
\begin{align*}
    \frac{d}{dt}\dot X_t^0
    =\zeta(t)D_XF(X_t^\ast;\theta_t^\ast)\dot X_t^0
    +\zeta(t)D_\theta F(X_t^\ast;\theta_t^\ast)(\vartheta_t-\theta_t^\ast),
    \qquad \dot X_0^0=0.
\end{align*}
The same compactness assumptions imply a uniform bound on $\dot X_t^0$ proportional to $\norm{\vartheta-\theta^\ast}_{L^1}$ or $\norm{\vartheta-\theta^\ast}_{L^2}$ on the local chart.  Since $DH$ and $\nabla_Z\ell$ are uniformly bounded on the relevant compact logit set, dominated convergence justifies differentiating the terminal cost under the expectation:
\begin{align*}
    \left.\frac{d}{d\alpha}\mcJ(\theta^\alpha)\right|_{\alpha=0}
    =\E\left[\inner{P_T^\ast}{\dot X_T^0}\right].
\end{align*}
Using \eqref{eq:pmp-adjoint} and the first-variation equation, one has
\begin{align*}
    \frac{\dd}{\dd t}\inner{P_t^\ast}{\dot X_t^0}
    &=\inner{\dot P_t^\ast}{\dot X_t^0}+\inner{P_t^\ast}{\frac{\dd}{\dd t}\dot X_t^0}\\
    &=-\zeta(t)\inner{D_XF(X_t^\ast;\theta_t^\ast)^\top P_t^\ast}{\dot X_t^0}
      +\zeta(t)\inner{P_t^\ast}{D_XF(X_t^\ast;\theta_t^\ast)\dot X_t^0}\\
    &\quad +\zeta(t)\inner{P_t^\ast}{D_\theta F(X_t^\ast;\theta_t^\ast)(\vartheta_t-\theta_t^\ast)}\\
    &=\zeta(t)\inner{P_t^\ast}{D_\theta F(X_t^\ast;\theta_t^\ast)(\vartheta_t-\theta_t^\ast)}.
\end{align*}
Since $\dot X_0^0=0$, integrating from $0$ to $T$ and then taking expectation over $X_0\sim\mu_X$ gives
\begin{align*}
    \E\left[\inner{P_T^\ast}{\dot X_T^0}\right]
    =\int_0^T\zeta(t)\E\left[
    \inner{P_t^\ast}{D_\theta F(X_t^\ast;\theta_t^\ast)(\vartheta_t-\theta_t^\ast)}
    \right]\dd t.
\end{align*}
Since $\theta^\ast$ is a local minimizer and $\theta^\alpha$ is feasible for $\alpha\ge0$ small, this derivative is nonnegative, proving \eqref{eq:pmp-vi}--\eqref{eq:pmp-vi-expanded}.

It remains to justify the two localized consequences.  Suppose first that $\calU=L^\infty([0,T];\Theta)$.  Fix $\widehat\theta\in\Theta$, a Lebesgue point $t_0$ of the functions involved, and a small interval
\begin{align*}
    I_h(t_0)=[t_0,t_0+h]\cap[0,T].
\end{align*}
Define the needle control
\begin{align*}
    \theta_t^{h,\widehat\theta}=\begin{cases}
    \widehat\theta, & t\in I_h(t_0),\\
    \theta_t^\ast, & t\notin I_h(t_0).
    \end{cases}
\end{align*}
The standard needle-expansion of the terminal variation gives
\begin{align}
\label{eq:needle-expansion-hamiltonian}
    \mcJ(\theta^{h,\widehat\theta})-\mcJ(\theta^\ast)
    =\int_{I_h(t_0)}
    \left[
    \overline{\mathcal H}_{\theta^\ast}(t,\widehat\theta)
    -\overline{\mathcal H}_{\theta^\ast}(t,\theta_t^\ast)
    \right]\dd t+o(|I_h(t_0)|),
\end{align}
where the remainder is uniform for $\widehat\theta$ in compact subsets of $\Theta$.  Since $\theta^\ast$ is locally optimal, the left-hand side is nonnegative for all sufficiently small $h$.  Dividing \eqref{eq:needle-expansion-hamiltonian} by $|I_h(t_0)|$ and using the Lebesgue differentiation theorem yields
\begin{align*}
    \overline{\mathcal H}_{\theta^\ast}(t_0,\widehat\theta)
    -\overline{\mathcal H}_{\theta^\ast}(t_0,\theta_{t_0}^\ast)
    \ge 0
\end{align*}
for almost every $t_0$ and every fixed $\widehat\theta\in\Theta$.  A countability argument on a dense subset of $\Theta$, followed by continuity in $\theta$, gives \eqref{eq:pmp-averaged-hamiltonian-min} for all $\theta\in\Theta$ and almost every $t$.

If $\theta_t^\ast$ lies in the interior of $\Theta$ for almost every $t$, then localized signed perturbations are feasible.  For any $\eta\in L^\infty([0,T];\R^{d_\theta})$ with support contained in times where $\operatorname{dist}(\theta_t^\ast,\partial\Theta)>0$, the controls $\theta^\ast\pm\alpha\eta$ are admissible for all sufficiently small $\alpha>0$.  Applying \eqref{eq:pmp-vi} to both signs gives
\begin{align*}
    \int_0^T
    \left\langle
    \nabla_\theta\overline{\mathcal H}_{\theta^\ast}(t,\theta_t^\ast),
    \eta_t
    \right\rangle\dd t=0.
\end{align*}
Taking $\eta_t=\xi\mathbf 1_E(t)$ for an arbitrary vector $\xi\in\R^{d_\theta}$ and measurable set $E$, and then applying Lebesgue differentiation, gives
\begin{align*}
    \left\langle
    \nabla_\theta\overline{\mathcal H}_{\theta^\ast}(t,\theta_t^\ast),
    \xi
    \right\rangle=0
    \qquad\text{for every }\xi\in\R^{d_\theta}
\end{align*}
for almost every interior time $t$.  This is precisely \eqref{eq:pmp-stationarity}.

\end{proof}

\begin{remark}[Hamiltonian form and boundary points]
The Hamiltonian condition \eqref{eq:pmp-averaged-hamiltonian-min} is the boundary-aware statement.  The equation \eqref{eq:pmp-stationarity} should only be used on times at which the optimal control lies in the interior of $\Theta$.  The label is not part of the transported state or of the Hamiltonian; it enters only through the terminal adjoint \eqref{eq:pmp-terminal}, which contains the cross-entropy softmax residual $\softmax(H(X_T))-Y(X_0)$.
When the averaged Hamiltonian minimizer is unique, the feedback form of the optimal control is
\begin{align}
\label{eq:pmp-law-selector}
    \theta_t^\ast=\Theta^\ast(\lambda_t^\ast),\qquad
    \Theta^\ast(\lambda)\in\arg\min_{\theta\in\Theta}\int \zeta(t)\inner{p}{F(x;\theta)}\,\lambda(\dd x,\dd p),
\end{align}
where $\lambda_t^\ast:=\calL(X_t^\ast,P_t^\ast)$ is the joint law of the state and adjoint. Thus the optimal control at time $t$ depends on the population joint distribution $\calL(X_t^\ast,P_t^\ast)$, not on an individual particlewise pair $(X_t^\ast(x),P_t^\ast(x))$.
\end{remark}

\section{Conditional stability results}
\label{sec:conditional-stability-results}

This section is a conditional toolbox.  The results below do not assert that compactness, strict separation, quadratic growth, PL inequalities, or nondegenerate Hessians automatically hold for Transformer objectives.  They state what follows once such assumptions have been verified on a chosen admissible class or basin.  This separation is meant to distinguish the proved approximation estimates in \Cref{sec:finite-depth-approximation} from structural consequences that require additional hypotheses.

\subsection{Existence and compactness of continuous controls}
\label{sec:full-continuous-control}

This subsection records several structural facts about the full continuous-time control problem
\begin{align}
\label{eq:full-continuous-value}
    V=\inf_{\theta\in\calU}\mcJ(\theta),
\end{align}
where $\calU$ is a class of admissible measurable controls with values in $\Theta$.  The results are intentionally stated in an abstract metric-control form, since the Transformer vector field is generally nonconvex in the layer parameter.  Thus existence and stability of minimizers follow from compactness, continuity, and separation assumptions, rather than from convexity of the loss landscape.

Let $d_{\calU}$ be a metric on $\calU$.  Typical choices are the $L^1$-metric,
\begin{align*}
    d_{1}(\theta,\vartheta)=\int_0^T\norm{\theta_t-\vartheta_t}\dd t,
\end{align*}
or the uniform metric $d_\infty$ when controls are continuous.  The estimates below only require that $\mcJ$ be continuous with respect to the chosen metric and that finite-depth controls approximate continuous controls in that metric or directly in objective value.

\begin{theorem}[Existence of global minimizers]
\label{thm:existence-global-minimizer}
Assume that $(\calU,d_{\calU})$ is compact and that $\theta\mapsto\mcJ(\theta)$ is continuous on $\calU$.  Then the full continuous control problem \eqref{eq:full-continuous-value} admits at least one global minimizer.  That is, there exists $\theta^\ast\in\calU$ such that
\begin{align*}
    \mcJ(\theta^\ast)=V.
\end{align*}
If, in addition, $\mcJ$ is strictly separated at $\theta^\ast$, in the sense that
\begin{align}
\label{eq:strict-global-separation}
    \mcJ(\theta)>\mcJ(\theta^\ast)
    \qquad \text{for every }\theta\in\calU\setminus\{\theta^\ast\},
\end{align}
then the global minimizer is unique.
\end{theorem}

\begin{proof}
Since $(\calU,d_{\calU})$ is compact and $\mcJ$ is continuous, by the Weierstrass extreme-value theorem (see, for example, \citet[Theorem~4.16]{Rudin1976}), there exists $\theta^\ast\in\calU$ such that
\begin{align*}
    \mcJ(\theta^\ast)=\min_{\theta\in\calU}\mcJ(\theta)=V.
\end{align*}
This proves existence.  If \eqref{eq:strict-global-separation} holds and $\widetilde\theta$ is another global minimizer, then
\begin{align*}
    \mcJ(\widetilde\theta)=\mcJ(\theta^\ast).
\end{align*}
The strict separation property forces $\widetilde\theta=\theta^\ast$.  Hence the global minimizer is unique.
\end{proof}

\begin{remark}[Local strict separation]
If the strict separation condition \eqref{eq:strict-global-separation} is replaced by local strict separation, then one obtains only local uniqueness.  More precisely, if there exists $r>0$ such that
\begin{align*}
    \mcJ(\theta)>\mcJ(\theta^\ast)
    \qquad\text{for every }0<d_{\calU}(\theta,\theta^\ast)<r,
\end{align*}
then no other minimizer can lie in the punctured ball $B_r(\theta^\ast)\setminus\{\theta^\ast\}$.  However, this does not rule out other global minimizers outside $B_r(\theta^\ast)$.  Thus local strict separation is sufficient for stability of a chosen local basin, but it is not sufficient for global uniqueness.  This distinction is important for Transformer-type objectives, which are generally nonconvex and may have multiple disconnected basins.
\end{remark}

\begin{remark}[When compactness is reasonable]\label{rem:compact}
The compactness assumption on $\calU$ can be realized in several standard ways.  If $\Theta$ is compact and $\calU$ is restricted to an equicontinuous family of controls, for example controls with a common Lipschitz bound, then Arzela--Ascoli compactness gives compactness in the uniform topology; see \citet[Theorem~7.25]{Rudin1976}.  A similar compactness mechanism applies to bounded-variation controls.  If $\Theta$ is compact and the controls have uniformly bounded total variation, Helly's selection principle yields a subsequence converging pointwise at every continuity point of the limit and strongly in $L^1$ after passing to a subsequence, provided the limit is interpreted as a bounded-variation control.  Thus a bounded-variation admissible class can be compact in an $L^1$ or pointwise-a.e. topology compatible with this selection principle.  These restrictions are natural when one wants the depth profile $t\mapsto\theta_t$ to vary smoothly across layers.

For the larger class $L^\infty([0,T];\Theta)$, compactness is not expected in the strong $L^1$ or $L^\infty$ topology.  One may instead work in a weak-* topology or pass to relaxed controls, where a control is represented by a time-dependent probability measure on $\Theta$.  Such compactifications are common in control theory, but they require a corresponding continuity statement for the control-to-state and control-to-risk maps.  The present analysis does not rely on a specific compactification: whenever a chosen topology makes $\calU$ compact and $\theta\mapsto\mcJ(\theta)$ continuous, the existence proof in \Cref{thm:existence-global-minimizer} applies verbatim.  
\end{remark}

\subsection{Local quadratic growth and stability of local minimizers}

After establishing existence under compactness, we turn to the stability of selected minimizers.  The following definition records the local growth property that turns objective-value convergence into convergence of controls.
\begin{definition}[Isolated local minimizer and quadratic growth]
A control $\theta^\ast\in\calU$ is a local minimizer if there exists $r_0>0$ such that
\begin{align*}
    \mcJ(\theta^\ast)\le \mcJ(\theta)
    \qquad \text{whenever }d_{\calU}(\theta,\theta^\ast)<r_0.
\end{align*}
It satisfies a local quadratic-growth condition if there are constants $\lambda>0$ and $r_0>0$ such that
\begin{align}
\label{eq:quadratic-growth}
    \mcJ(\theta)\ge \mcJ(\theta^\ast)+\frac{\lambda}{2}d_{\calU}(\theta,\theta^\ast)^2
    \qquad \text{whenever }d_{\calU}(\theta,\theta^\ast)<r_0.
\end{align}
\end{definition}

\begin{definition}[Local Hilbert model for controls]\label{def:Hilbert-model}
We say that $\calU$ is locally modeled on a Hilbert space $\mathcal H_{\calU}$ near $\theta^\ast$ if there is a neighborhood $\mathcal O$ of $\theta^\ast$ such that every admissible control in $\mathcal O$ can be written as $\theta^\ast+h$ with $h\in\mathcal H_{\calU}$, and the control metric is induced by the Hilbert norm:
\begin{align*}
    d_{\calU}(\theta^\ast+h_1,\theta^\ast+h_2)=\norm{h_1-h_2}_{\calU}.
\end{align*}
For unconstrained controls one may take $\mathcal H_{\calU}=L^2([0,T];\R^{d_\theta})$ locally, after identifying a control $\theta$ near $\theta^\ast$ with its perturbation $h=\theta-\theta^\ast$.  In this case the local metric is $d_{\calU}(\theta^\ast+h_1,\theta^\ast+h_2)=\|h_1-h_2\|_{L^2}$.  
\end{definition}

\begin{theorem}[Sufficient condition for local quadratic growth]
\label{thm:local-qg-sufficient}
Suppose that $\calU$ is locally modeled on a Hilbert space $\mathcal H_{\calU}$ near $\theta^\ast$ in the sense of Definition \ref{def:Hilbert-model}. Suppose that $\theta^\ast$ is an interior local minimizer, that $\mcJ$ is twice Fr\'echet differentiable in a neighborhood of $\theta^\ast$, and that the second variation is locally coercive: there exist $\lambda_0>0$ and $r_0>0$ such that
\begin{align}
\label{eq:second-variation-coercive}
    D^2\mcJ(\theta)[h,h]
    \ge \lambda_0\norm{h}_{\calU}^2
\end{align}
for every $\theta\in B_{r_0}(\theta^\ast)$ and every admissible direction $h$.  Then $\theta^\ast$ satisfies the local quadratic-growth condition \eqref{eq:quadratic-growth} with any $\lambda<\lambda_0$, after possibly reducing $r_0$.
\end{theorem}

\begin{proof}
Since $\theta^\ast$ is an interior local minimizer and $\mcJ$ is differentiable, the first-order condition gives $D\mcJ(\theta^\ast)=0$.  For $\theta$ close to $\theta^\ast$, write $h=\theta-\theta^\ast$.  Taylor's formula with integral remainder gives
\begin{align*}
    \mcJ(\theta)-\mcJ(\theta^\ast)
    =\int_0^1(1-s)D^2\mcJ(\theta^\ast+s h)[h,h]\dd s.
\end{align*}
If $\theta^\ast+s h\in B_{r_0}(\theta^\ast)$, the coercivity condition \eqref{eq:second-variation-coercive} yields
\begin{align*}
    \mcJ(\theta)-\mcJ(\theta^\ast)
    \ge \int_0^1(1-s)\lambda_0\norm{h}_{\calU}^2\dd s
    =\frac{\lambda_0}{2}d_{\calU}(\theta,\theta^\ast)^2.
\end{align*}
Thus \eqref{eq:quadratic-growth} holds, up to replacing $\lambda_0$ by any smaller constant if the ball is reduced.
\end{proof}

In the Transformer control problem, condition \eqref{eq:second-variation-coercive} is an additional nondegeneracy assumption, not an automatic consequence of cross-entropy training.  It may fail because of parameter symmetries, flat directions, saturation of the softmax, or non-identifiability of different controls producing the same terminal representation.

To see why it is not automatic, set
\begin{align*}
    G(\theta;X_0)=H(\Phi_T^\theta(X_0))\in\R^{L\times d_v}.
\end{align*}
The cross-entropy loss is convex in the logit variable $Z$, but the control-to-logit map $\theta\mapsto G(\theta;X_0)$ is nonlinear and generally nonconvex.  Convexity of $\ell(Y,Z)$ in $Z$ therefore does not imply convexity, much less uniform convexity, in the control $\theta$.  More precisely, for an admissible direction $h$, the second variation has the schematic form
\begin{align}
\label{eq:hessian-decomposition}
    D^2\mcJ(\theta)[h,h]
    &=\E\left[\left\langle D^2_Z\ell(Y,G(\theta))[DG(\theta)[h]],DG(\theta)[h]\right\rangle\right]  \notag\\
    &\quad +\E\left[\left\langle \nabla_Z\ell(Y,G(\theta)),D^2G(\theta)[h,h]\right\rangle\right].
\end{align}
The first term is the positive Gauss--Newton part.  It vanishes in every first-order null direction satisfying $DG(\theta)[h]=0$.  Such directions are common in overparameterized Transformer parameterizations.  In the concrete toy dimensions $d_x=L=d_v=d_k=2$, take one attention head and a linear readout.  Suppose, for example, that the value and output factors enter one terminal hidden coordinate only through a scalar product $ab$ inside the effective matrix product $W_VW_O$.  At $a=b=1$, the perturbation $(\delta a,\delta b)=(1,-1)$ leaves this product unchanged to first order, since
\begin{align*}
    \delta(ab)=b\,\delta a+a\,\delta b=1\cdot1+1\cdot(-1)=0.
\end{align*}
Consequently the induced first-order terminal-logit variation vanishes in this direction, i.e. $DG(\theta)[h]=0$, even though the underlying parameters have changed.  Similarly, if $H(X)=XH_1$ with $H_1=\begin{pmatrix}1&0\\0&0\end{pmatrix}$ and a perturbation changes only the second hidden coordinate at terminal time, then $D H(X_T)[\delta X_T]=0$, so no positive cross-entropy curvature is seen in that direction.  The same degeneracy also appears in standard Transformer symmetries: permuting identical heads, splitting one head into two equivalent heads with compensating output weights, or refactorizing a feed-forward map $W_1W_2$ without changing its product can change the parameter vector while leaving the realized input-output map unchanged to first order.  These directions are genuine null directions for the risk unless one quotients out the symmetry or imposes an identifiability condition.  To rule this out one needs an observability or identifiability restriction on the readout and on the reachable terminal variations, such as a lower bound of the form
\begin{align*}
    \E\norm{\Pi_\perp D H(X_T^\theta)D\Phi_T^\theta[h]}^2\ge c\norm{h}_{\calU}^2
\end{align*}
for admissible directions $h$, modulo the unavoidable constant-logit directions; also see \eqref{eq:observability-lower-bound}.

Even when $DG(\theta)[h]\ne0$, the softmax Hessian may be poorly conditioned.  For one row, if $p=\softmax(z)$, then
\begin{align*}
    D^2_z\ell(y,z)=\operatorname{diag}(p)-pp^\top,
\end{align*}
up to the factor $1/L$ in the token-averaged loss.  The matrix $\operatorname{diag}(p)-pp^\top$ annihilates the constant-logit direction $\mathbf 1$ since adding the same constant to every vocabulary logit does not change the softmax.  It is positive only on the quotient space modulo $\operatorname{span}\{\mathbf 1\}$.  In the binary case $d_v=2$ with $z=(a,0)$, one has $p_1=e^a/(e^a+1)$ and the nonzero curvature is proportional to $p_1(1-p_1)$, which tends to zero as $a\to\infty$.  In our toy Transformer dimensions $d_x=L=d_v=d_k=2$, this corresponds to a two-token model in which each row logit is nearly saturated, for example $H(X)_{l,:}=(a,0)$ with $a\gg1$.  Then even a visible perturbation of the terminal hidden state produces almost no cross-entropy curvature.  To prevent this loss of curvature, one may impose a non-saturation condition on the relevant logits, for instance $p_j\in[\alpha,1-\alpha]$ for the active vocabulary coordinates and some $\alpha>0$.  Equivalent mechanisms include logit clipping, temperature control, label smoothing, or regularization terms that keep logits and parameters in a bounded region.  Without such a condition, the smallest positive eigenvalue of $\operatorname{diag}(p)-pp^\top$ on the quotient space can be arbitrarily close to zero, and therefore the positive Gauss--Newton part in \eqref{eq:hessian-decomposition} cannot yield a uniform lower bound.

The preceding discussion identifies the two main obstructions to coercivity: invisible logit directions and softmax saturation.  The next theorem states a concrete sufficient condition excluding these obstructions through observable logit curvature.
\begin{theorem}[Observable logit curvature implies local coercivity]
\label{thm:observable-logit-coercivity}
Let $G(\theta;X_0)=H(\Phi_T^\theta(X_0))$ be twice Fr\'echet differentiable in a neighborhood $\mathcal O$ of $\theta^\ast$.  Assume that the following three conditions hold on $\mathcal O$.
\begin{enumerate}
    \item There is $C_{\rm rem}\ge0$ such that, for every admissible direction $h$,
    \begin{align}
    \label{eq:second-order-remainder-bound}
        \left|\E\left[\left\langle \nabla_Z\ell(Y(X_0),G(\theta;X_0)),D^2G(\theta;X_0)[h,h]\right\rangle\right]\right|
        \le C_{\rm rem}\norm{h}_{\calU}^2.
    \end{align}
    \item There is $\kappa>0$ such that, for every relevant logit $Z$ and every perturbation $U\in\R^{L\times d_v}$,
    \begin{align}
    \label{eq:softmax-hessian-lower}
        \left\langle D_Z^2\ell(Y,Z)U,U\right\rangle
        \ge \kappa\norm{\Pi_{\perp}U}^2,
    \end{align}
    where $\Pi_{\perp}$ denotes row-wise projection onto the orthogonal complement of the constant-logit direction.  Explicitly, for $U\in\R^{L\times d_v}$,
    \begin{align*}
        (\Pi_{\perp}U)_{l,:}=U_{l,:}-\frac1{d_v}\left(\sum_{j=1}^{d_v}U_{l,j}\right)\mathbf 1_{d_v},
        \qquad l=0,\ldots,L-1.
    \end{align*}
    This removes the row-wise additive logit shift, which is invisible to the softmax.
    \item There is $c_{\rm obs}>0$ such that
    \begin{align}
    \label{eq:observability-lower-bound}
        \E_{X_0\sim\mu_X}\norm{\Pi_{\perp}DG(\theta;X_0)[h]}^2
        \ge c_{\rm obs}\norm{h}_{\calU}^2
    \end{align}
    for all admissible directions $h$.
\end{enumerate}
If $\kappa c_{\rm obs}>C_{\rm rem}$, then
\begin{align*}
    D^2\mcJ(\theta)[h,h]\ge (\kappa c_{\rm obs}-C_{\rm rem})\norm{h}_{\calU}^2
\end{align*}
for every $\theta\in\mathcal O$ and every admissible direction $h$.  In particular, the local coercivity condition \eqref{eq:second-variation-coercive} holds on a possibly smaller neighborhood of $\theta^\ast$ with any $\lambda_0<\kappa c_{\rm obs}-C_{\rm rem}$.
\end{theorem}

\begin{proof}
Using the chain rule for $\mcJ(\theta)=\E[\ell(Y(X_0),G(\theta;X_0))]$ gives the decomposition \eqref{eq:hessian-decomposition}.  By \eqref{eq:softmax-hessian-lower}, the Gauss--Newton term is bounded from below by
\begin{align*}
    \kappa\,\E\norm{\Pi_{\perp}DG(\theta;X_0)[h]}^2.
\end{align*}
The observability estimate \eqref{eq:observability-lower-bound} further bounds this by $\kappa c_{\rm obs}\norm{h}_{\calU}^2$.  The remaining second-order term is bounded below by $-C_{\rm rem}\norm{h}_{\calU}^2$ by \eqref{eq:second-order-remainder-bound}.  Combining these estimates yields the claimed lower bound.  Reducing the neighborhood if necessary gives the stated local coercivity condition.
\end{proof}

When observability alone is not available, explicit regularization can supply missing curvature on a restricted Hilbert chart.  The following theorem records this standard mechanism in the notation of the control problem.
\begin{theorem}[Regularization-enforced local quadratic growth]
\label{thm:regularization-coercivity}
Assume that $\calU$ is locally modeled on a Hilbert space and that the unregularized risk satisfies the local lower Hessian bound
\begin{align}
\label{eq:negative-curvature-bound}
    D^2\mcJ(\theta)[h,h]\ge -\beta\norm{h}_{\calU}^2
\end{align}
for all $\theta$ in a neighborhood of $\theta^\ast$ and all admissible directions $h$.  Define the regularized objective
\begin{align*}
    \mcJ_\gamma(\theta)=\mcJ(\theta)+\frac{\gamma}{2}\norm{\theta}_{\calU}^2,
\end{align*}
where $\gamma>\beta$.  Then
\begin{align*}
    D^2\mcJ_\gamma(\theta)[h,h]\ge (\gamma-\beta)\norm{h}_{\calU}^2.
\end{align*}
Consequently, every interior local minimizer of $\mcJ_\gamma$ satisfies a local quadratic-growth condition.
\end{theorem}

\begin{proof}
The quadratic regularizer has second variation $\gamma\norm{h}_{\calU}^2$.  Adding this to \eqref{eq:negative-curvature-bound} gives
\begin{align*}
    D^2\mcJ_\gamma(\theta)[h,h]
    =D^2\mcJ(\theta)[h,h]+\gamma\norm{h}_{\calU}^2
    \ge (\gamma-\beta)\norm{h}_{\calU}^2.
\end{align*}
Since $\gamma>\beta$, \Cref{thm:local-qg-sufficient} applies to the regularized objective and gives local quadratic growth.
\end{proof}

\begin{remark}[Quadratic regularization and random connection dropping]
The deterministic term $\frac{\gamma}{2}\norm{\theta}_{\calU}^2$ should not be identified literally with dropout or DropConnect, but it plays a closely related stabilizing role.  Dropout randomly removes activations during training, while DropConnect randomly removes weights or connections.  In simple generalized-linear or locally quadratic regimes, averaging over this multiplicative noise produces an additional data-dependent penalty that behaves like an adaptive quadratic regularizer; see, for example, \citet{HintonEtAl2012,WagerWangLiang2013,WanEtAl2013}.  Thus the regularized objective $\mcJ_\gamma$ can be viewed as a deterministic analytic analogue of noise-induced stabilization: it penalizes large control norms and can remove flat or negative-curvature directions when the penalty strength dominates the local nonconvexity bound.  The analogy is only qualitative in the present continuous-depth Transformer setting, since dropout changes the stochastic training objective and may introduce state-dependent regularization rather than the isotropic Hilbert penalty used in \Cref{thm:regularization-coercivity}.
\end{remark}

Here is a concrete setting in which the local quadratic-growth and PL assumptions can be verified rather than postulated. 
\begin{example}[A restricted architecture with checkable quadratic growth and PL]
\label{ex:frozen-feature-readout-qg}
 Freeze the continuous hidden-state dynamics, so that $X_T=\Phi_T(X_0)$ is independent of the trainable parameter in this example, and train only a finite-dimensional readout chart
\begin{align}
\label{eq:frozen-feature-readout-chart}
    H_a(X)=H_0(X)+\Psi(X)a,
    \qquad a\in\R^m,
\end{align}
where $\Psi(X)a\in\R^{L\times d_v}$ is linear in $a$.  Write
\begin{align*}
    Z_{l,a}(X_0):=H_{0,l}(X_T(X_0))+\Psi_l(X_T(X_0))a,
    \qquad
    p_{l,a}(X_0):=\softmax(Z_{l,a}(X_0)).
\end{align*}
The reduced objective is
\begin{align*}
    j(a)=\E\left[\frac1L\sum_{l=0}^{L-1}
    \left(\log\sum_{r=1}^{d_v}e^{Z_{l,a,r}(X_0)}
    -Y_l(X_0)^\top Z_{l,a}(X_0)\right)\right].
\end{align*}
For any direction $h\in\R^m$, differentiation under the expectation gives
\begin{align}
\label{eq:restricted-readout-first-variation}
    Dj(a)[h]
    &=\E\left[\frac1L\sum_{l=0}^{L-1}
    \bigl(p_{l,a}(X_0)-Y_l(X_0)\bigr)^\top
    \Psi_l(X_T(X_0))h\right],\\
\label{eq:restricted-readout-second-variation}
    D^2j(a)[h,h]
    &=\E\left[\frac1L\sum_{l=0}^{L-1}
    v_{l,h}(X_0)^\top
    \Bigl(\operatorname{diag}(p_{l,a}(X_0))-p_{l,a}(X_0)p_{l,a}(X_0)^\top\Bigr)
    v_{l,h}(X_0)\right],
\end{align}
where $v_{l,h}(X_0):=\Psi_l(X_T(X_0))h$.

Assume that the admissible readout chart is restricted to the mean-zero logit subspace,
\begin{align*}
    v_{l,h}(X_0)=\Pi_0v_{l,h}(X_0),
    \qquad
    \Pi_0:=I-d_v^{-1}\mathbf 1\mathbf 1^\top,
\end{align*}
for every $l,h$, and $X_0$.  This removes the softmax shift invariance.  Assume also that the logits are uniformly bounded,
\begin{align*}
    \|Z_{l,a}(X_0)\|_\infty\le B_Z
    \qquad\text{for all }l,a,X_0
\end{align*}
on the chart under consideration.  Then every softmax component satisfies
\begin{align*}
    p_{l,a,j}(X_0)\ge p_{\min}(B_Z,d_v):=\frac{e^{-B_Z}}{d_v e^{B_Z}}
    =\frac{e^{-2B_Z}}{d_v}.
\end{align*}
For $v\in\mathbf 1^\perp$, the softmax Hessian obeys the covariance identity
\begin{align*}
    v^\top(\operatorname{diag}(p)-pp^\top)v
    =\sum_{i<j}p_ip_j(v_i-v_j)^2
    \ge d_v p_{\min}(B_Z,d_v)^2\|v\|^2.
\end{align*}
If the averaged feature Gram matrix is nondegenerate, namely
\begin{align}
\label{eq:restricted-feature-gram}
    \E\left[\frac1L\sum_{l=0}^{L-1}
    \norm{\Pi_0\Psi_l(X_T(X_0))h}^2\right]
    \ge c_\Psi\norm{h}^2
    \qquad\text{for all }h\in\R^m,
\end{align}
then \eqref{eq:restricted-readout-second-variation} yields the uniform curvature bound
\begin{align}
\label{eq:restricted-readout-curvature}
    D_a^2 j(a)[h,h]
    \ge \lambda_{\rm ro}\norm{h}^2,
    \qquad
    \lambda_{\rm ro}:=d_v p_{\min}(B_Z,d_v)^2c_\Psi>0.
\end{align}
Thus the restricted readout objective is strongly convex on this chart.  

A simple concrete choice makes the feature-Gram condition \eqref{eq:restricted-feature-gram} transparent.  Suppose $m=d_v-1$ and choose a fixed matrix $B_0\in\R^{d_v\times(d_v-1)}$ whose columns form an orthonormal basis of $\mathbf 1^\perp$.  If, for every token $l$, the readout feature map is
\begin{align}
\label{eq:specific-psi-example}
    \Psi_l(X)h=B_0h,
    \qquad h\in\R^{d_v-1},
\end{align}
then $\Pi_0\Psi_l(X)h=B_0h$ and
\begin{align*}
    \E\left[\frac1L\sum_{l=0}^{L-1}\norm{\Pi_0\Psi_l(X_T(X_0))h}^2\right]
    =\norm{h}^2.
\end{align*}
Thus the averaged feature Gram matrix is nondegenerate with $c_\Psi=1$.  More generally, any feature family whose averaged matrix $\E[L^{-1}\sum_l \Psi_l^\top\Pi_0\Psi_l]$ has a positive minimum eigenvalue satisfies the same condition.

If $a^\ast$ is its minimizer, then
\begin{align*}
    j(a)-j(a^\ast)\ge \frac{\lambda_{\rm ro}}2\norm{a-a^\ast}^2
\end{align*}
and the standard strong-convexity argument gives the PL inequality
\begin{align*}
    \frac12\norm{\nabla j(a)}^2\ge \lambda_{\rm ro}\bigl(j(a)-j(a^\ast)\bigr)
\end{align*}
on any convex sublevel set inside the chart.  Adding a ridge penalty $\gamma\|a\|^2/2$ changes the curvature lower bound to $\lambda_{\rm ro}+\gamma$.  The example is deliberately restricted--the hidden dynamics and feature map are frozen--but it gives a nontrivial architecture class in which the curvature assumptions used in the stability and local SGD statements can be checked directly from bounded logits and a feature-Gram nondegeneracy condition.
\end{example}

The next result uses local quadratic growth only inside a chosen basin.  It therefore gives a local stability statement for finite-depth approximations rather than a claim about global minimizers.
\begin{theorem}[Stability of isolated local minimizers under finite-depth approximation]
\label{thm:local-minimizer-stability}
Let $\theta^\ast\in\calU$ be a local minimizer satisfying the quadratic-growth condition \eqref{eq:quadratic-growth}.  Let $\calA_\eps\subset\calU$ be a finite-depth admissible class and suppose that, on the ball
\begin{align*}
    B_{r_0}(\theta^\ast)=\{\theta\in\calU:d_{\calU}(\theta,\theta^\ast)<r_0\},
\end{align*}
the continuous risk and a finite-depth surrogate $\mcJ_\eps$ satisfy the uniform approximation estimate
\begin{align}
\label{eq:local-uniform-approx}
    \sup_{\theta\in\calA_\eps\cap B_{r_0}(\theta^\ast)}
    \abs{\mcJ_\eps(\theta)-\mcJ(\theta)}\le \Gamma_\eps,
\end{align}
where $\Gamma_\eps\to0$ as $\eps\to0$.  Define the local discrete class
\begin{align*}
    \calA_{\eps,r_0}(\theta^\ast)
    :=\calA_\eps\cap B_{r_0}(\theta^\ast).
\end{align*}
Let $\widehat\theta_\eps\in\calA_{\eps,r_0}(\theta^\ast)$ be an $\eta_\eps$-minimizer of $\mcJ_\eps$ on this local discrete class, meaning that
\begin{align*}
    \mcJ_\eps(\widehat\theta_\eps)
    \le \inf_{\theta\in\calA_{\eps,r_0}(\theta^\ast)}\mcJ_\eps(\theta)+\eta_\eps.
\end{align*}
Assume that $\theta^\ast\in\calA_\eps$ or, more generally, that there exists $\theta_\eps^\ast\in\calA_{\eps,r_0}(\theta^\ast)$ with
\begin{align}
\label{eq:local-recovery-control}
    \mcJ(\theta_\eps^\ast)\le \mcJ(\theta^\ast)+\alpha_\eps,
\end{align}
where $\alpha_\eps\to0$ as $\eps\to0$.  Then, for all sufficiently small $\eps$,
\begin{align}
\label{eq:local-minimizer-distance}
    d_{\calU}(\widehat\theta_\eps,\theta^\ast)^2
    \le \frac{2}{\lambda}\bigl(2\Gamma_\eps+\eta_\eps+\alpha_\eps\bigr).
\end{align}
Consequently, $\widehat\theta_\eps\to\theta^\ast$ whenever $\Gamma_\eps+\eta_\eps+\alpha_\eps\to0$ as $\eps\to0$.
\end{theorem}

\begin{proof}
By the $\eta_\eps$-minimality of $\widehat\theta_\eps$ for $\mcJ_\eps$ on $\calA_{\eps,r_0}(\theta^\ast)$,
\begin{align*}
    \mcJ_\eps(\widehat\theta_\eps)
    \le \mcJ_\eps(\theta_\eps^\ast)+\eta_\eps.
\end{align*}
Using the uniform approximation estimate \eqref{eq:local-uniform-approx} on both sides and then the recovery estimate \eqref{eq:local-recovery-control}, we obtain
\begin{align*}
    \mcJ(\widehat\theta_\eps)
    &\le \mcJ_\eps(\widehat\theta_\eps)+\Gamma_\eps \\
    &\le \mcJ_\eps(\theta_\eps^\ast)+\eta_\eps+\Gamma_\eps \\
    &\le \mcJ(\theta_\eps^\ast)+\eta_\eps+2\Gamma_\eps \\
    &\le \mcJ(\theta^\ast)+\alpha_\eps+\eta_\eps+2\Gamma_\eps.
\end{align*}
On the other hand, the quadratic-growth condition gives
\begin{align*}
    \mcJ(\widehat\theta_\eps)
    \ge \mcJ(\theta^\ast)+\frac{\lambda}{2}d_{\calU}(\widehat\theta_\eps,\theta^\ast)^2.
\end{align*}
Combining the two inequalities proves \eqref{eq:local-minimizer-distance}.
\end{proof}

To use the local stability theorem in practice, one needs a uniform compact tube on which the constants in the Euler and loss estimates are valid.  The following proposition gives a simple sufficient condition.
\begin{proposition}[Uniform compactness of local finite-depth trajectories]
\label{prop:uniform-local-compactness}
Assume that the input support $K_{\rm in}:=\operatorname{supp}\mu_X$ is compact and that the finite-depth class satisfies $\calA_\eps\subset\Theta^M$ with the same compact parameter set $\Theta$ for every sufficiently small $\eps$.  Suppose that there is $B_F<\infty$ such that
\begin{align*}
    \norm{F(X;\theta)}\le B_F
\end{align*}
for all $\theta\in\Theta$ and all $X$ in the compact tube
\begin{align*}
    K_T:=\left\{X\in\R^{L\times d_x}:\operatorname{dist}(X,K_{\rm in})\le T\zeta_+B_F+1\right\}.
\end{align*}
Then, for every $\thetaeps\in\calA_\eps\cap B_{r_0}(\theta^\ast)$, all layer parameters take values in the same compact set $\Theta$, and all continuous and discrete trajectories starting from $K_{\rm in}$ remain in the common compact state set $K_T$ on $[0,T]$, provided $\eps$ is sufficiently small.
\end{proposition}

\begin{proof}
The parameter statement is immediate from $\calA_\eps\subset\Theta^M$: each admissible layer sequence has coordinates $\theta_k\in\Theta$, independently of $\eps$ and independently of the local ball.  For the continuous flow, for any $X_0\in K_{\rm in}$,
\begin{align*}
    \norm{X_t-X_0}
    \le \int_0^t \zeta(s)\norm{F(X_s;\theta_s)}\dd s
    \le T\zeta_+B_F
\end{align*}
as long as the trajectory stays in $K_T$.  The right-hand side is strictly smaller than the tube radius $T\zeta_+B_F+1$, so the usual continuity-bootstrap argument prevents exit from $K_T$ before time $T$.  For the discrete recursion,
\begin{align*}
    \norm{X_k^\eps-X_0}
    \le \sum_{m=0}^{k-1}\eps\zeta(t_m)\norm{F(X_m^\eps;\theta_m)}
    \le T\zeta_+B_F,
\end{align*}
again as long as the iterates remain in $K_T$.  Since the bound is strictly inside the tube radius, the same induction/bootstrap argument shows that all iterates remain in $K_T$.  Therefore the local class uses one compact parameter set and one compact state set uniformly.
\end{proof}

\begin{remark}[When the local uniform approximation holds]
In the Transformer setting of this paper, \eqref{eq:local-uniform-approx} holds for the population finite-depth surrogate under the uniform assumptions used in \Cref{thm:pathwise}.  Indeed, Proposition \ref{prop:uniform-local-compactness} gives a concrete sufficient condition ensuring that every control in $\calA_\eps\cap B_{r_0}(\theta^\ast)$ takes values in the same compact parameter set, all trajectories remain in the same compact state set $K_T$, and the constants in Assumptions \ref{ass:controls}, \ref{ass:field}, and \ref{ass:readout} are uniform over the local class.  Then \Cref{cor:population-discretization} gives
\begin{align*}
    \sup_{\theta\in\calA_\eps\cap B_{r_0}(\theta^\ast)}
    \abs{\mcJ_\eps(\theta)-\mcJ(\theta)}\le C_T\eps.
\end{align*}
Thus one can take $\Gamma_\eps=C_T\eps$ for population risks.  If $\mcJ_\eps$ is replaced by the empirical finite-depth risk $\mcJ_{N,\eps}$ in \eqref{eq:empirical-risk}, the uniform approximation also contains a sampling term.  For a single fixed control this sampling term is given by \eqref{eq:fixed-control-error}; for finite local classes it is controlled by the finite-class bound \eqref{eq:finite-class-error}; and for infinite local classes it is controlled by the entropy estimate \eqref{eq:metric-entropy-error}.
\end{remark}

\subsection{Global quadratic growth and global minimizer stability}

We now contrast the preceding local analysis with a global stability statement.  Since global quadratic growth is a strong assumption for Transformer objectives, the next theorem first records a classical sufficient condition.
\begin{theorem}[Sufficient condition for global quadratic growth]
\label{thm:global-qg-sufficient}
Let $\calU$ be a convex subset of a Hilbert space with metric induced by the Hilbert norm.  Suppose that $\mcJ$ is Fr\'echet differentiable and globally $\lambda$-strongly convex on $\calU$, namely, for every $\theta,\vartheta\in\calU$ and every $s\in[0,1]$,
\begin{align}
\label{eq:strong-convexity-control}
    \mcJ((1-s)\theta+s\vartheta)
    \le (1-s)\mcJ(\theta)+s\mcJ(\vartheta)
    -\frac{\lambda}{2}s(1-s)d_{\calU}(\theta,\vartheta)^2.
\end{align}
If $\theta^\ast$ is a global minimizer, then it satisfies the global quadratic-growth condition \eqref{eq:global-growth}.

\end{theorem}

\begin{proof}
Strong convexity implies the first-order inequality
\begin{align*}
    \mcJ(\theta)\ge \mcJ(\vartheta)+\inner{D\mcJ(\vartheta)}{\theta-\vartheta}
    +\frac{\lambda}{2}d_{\calU}(\theta,\vartheta)^2
\end{align*}
for all $\theta,\vartheta\in\calU$ for which the line segment remains in $\calU$.  Taking $\vartheta=\theta^\ast$ and using the variational inequality for a global minimizer over the convex set,
\begin{align*}
    \inner{D\mcJ(\theta^\ast)}{\theta-\theta^\ast}\ge0,
\end{align*}
we obtain
\begin{align*}
    \mcJ(\theta)\ge \mcJ(\theta^\ast)+\frac{\lambda}{2}d_{\calU}(\theta,\theta^\ast)^2,
\end{align*}
which is \eqref{eq:global-growth}.
\end{proof}

For the Transformer control objective, global strong convexity is generally not expected, because the flow map depends nonlinearly on the controls and the parameterization has symmetries.  Therefore global quadratic growth should be viewed as a structural stability assumption or as a property of a restricted or regularized control class, rather than as an automatic feature of the model.
More explicitly, condition \eqref{eq:strong-convexity-control} is not a consequence of using cross-entropy loss.  Cross-entropy is convex in the logit variable, but the control-to-logit map $\theta\mapsto H(\Phi_T^\theta(X_0))$ is nonlinear since the state equation itself depends on $\theta$.  The composition of a convex function with a nonlinear non-affine map need not be convex.  In addition, different controls may generate the same vector field on the data-supporting region, and attention or feed-forward parameterizations may have scaling and permutation symmetries.  These effects create flat directions incompatible with global strong convexity.

Checkable sufficient conditions for \eqref{eq:strong-convexity-control} are correspondingly restrictive.  One may obtain it if the admissible control class is restricted to a convex finite-dimensional chart, the terminal logit map is affine or uniformly close to affine on that chart, the induced design/operator matrix has a uniform lower singular-value bound, and the logits remain in a region where the cross-entropy Hessian is uniformly positive on the identifiable subspace.  Alternatively, adding a control regularization term
\begin{align*}
    \mcJ_\gamma(\theta)=\mcJ(\theta)+\frac{\gamma}{2}\int_0^T\norm{\theta_t}^2\dd t
\end{align*}
can make the regularized objective strongly convex on a convex restricted class if the unregularized part has Hessian bounded below by $-\beta I$ and $\gamma>\beta$.  Without such identifiability, curvature, and regularization assumptions, \eqref{eq:strong-convexity-control} should be treated as a hypothesis for stability analysis rather than a general property of Transformer training.

Under the global growth assumption, value convergence controls the distance to the unique global minimizer.  The next theorem is the global analogue of the local finite-depth stability result.
\begin{theorem}[Global minimizer stability]
\label{thm:global-minimizer-stability}
Assume that $(\calU,d_{\calU})$ is compact and that $\theta\mapsto\mcJ(\theta)$ is continuous on $\calU$.  Suppose that the full continuous control problem admits a unique global minimizer $\theta^\ast$ and that $\theta^\ast$ satisfies the global quadratic-growth condition
\begin{align}
\label{eq:global-growth}
    \mcJ(\theta)\ge \mcJ(\theta^\ast)+\frac{\lambda}{2}d_{\calU}(\theta,\theta^\ast)^2
    \qquad \text{for all }\theta\in\calU.
\end{align}
Let $\calA_\eps\subset\calU$ be a finite-depth class and assume that the uniform and recovery estimates
\begin{align*}
    \sup_{\theta\in\calA_\eps}\abs{\mcJ_\eps(\theta)-\mcJ(\theta)}\le \Gamma_\eps,
    \qquad
    \inf_{\theta\in\calA_\eps}\mcJ(\theta)
    \le \mcJ(\theta^\ast)+\alpha_\eps
\end{align*}
hold with $\Gamma_\eps\to0$ and $\alpha_\eps\to0$ as $\eps\to0$.  If $\widehat\theta_\eps$ is an $\eta_\eps$-minimizer of $\mcJ_\eps$ over $\calA_\eps$, then
\begin{align}
\label{eq:global-minimizer-distance}
    d_{\calU}(\widehat\theta_\eps,\theta^\ast)^2
    \le \frac{2}{\lambda}\bigl(2\Gamma_\eps+\eta_\eps+\alpha_\eps\bigr).
\end{align}
In particular, discrete approximate global minimizers $\widehat\theta_\eps$ converge to $\theta^\ast$ if $\Gamma_\eps+\eta_\eps+\alpha_\eps\to0$ as $\eps\to0$.
\end{theorem}

\begin{proof}
The proof is the same as the local argument, but now the quadratic-growth condition holds on all of $\calU$.  Choose $\theta_\eps^\ast\in\calA_\eps$ such that
\begin{align*}
    \mcJ(\theta_\eps^\ast)\le \mcJ(\theta^\ast)+\alpha_\eps+o(1),
\end{align*}
and absorb the harmless $o(1)$ into $\alpha_\eps$.  The approximate minimality of $\widehat\theta_\eps$ and the uniform approximation bound give
\begin{align*}
    \mcJ(\widehat\theta_\eps)
    \le \mcJ(\theta^\ast)+2\Gamma_\eps+\eta_\eps+\alpha_\eps.
\end{align*}
Combining this with \eqref{eq:global-growth} proves \eqref{eq:global-minimizer-distance}.
\end{proof}

\subsection{From the continuous optimizer to a discrete Transformer}

The preceding results suggest the following principled route from the continuous problem to a finite-depth Transformer.  First, solve or approximate the continuous first-order condition in Proposition \ref{prop:pmp}.  In the unconstrained case, the population control gradient is
\begin{align}
\label{eq:control-gradient}
    G_\theta(t)=\zeta(t)\,\E\left[D_\theta F(X_t^\theta;\theta_t)^\top P_t^\theta\right].
\end{align}
A formal gradient descent in the control variable is
\begin{align}
\label{eq:control-gradient-flow}
    \partial_s\theta_s(t)=-G_{\theta_s}(t),
\end{align}
with projection onto $\Theta$ if the control set is constrained.  Along smooth solutions of this auxiliary optimization flow, the first-variation formula in the proof of Proposition \ref{prop:pmp} yields
\begin{align*}
    \frac{\dd}{\dd s}\mcJ(\theta_s)
    =-\int_0^T\norm{G_{\theta_s}(t)}^2\dd t\le0,
\end{align*}
so stationary points of the control-gradient flow satisfy the Pontryagin stationarity condition.

Second, after a continuous candidate $\theta^\ast$ has been obtained, one builds a finite-depth Transformer by sampling or averaging the control on the depth grid:
\begin{align}
\label{eq:continuous-to-discrete-control}
    \theta_k^\eps=\theta^\ast(t_k)
    \quad \text{or} \quad
    \theta_k^\eps=\frac1\eps\int_{t_k}^{t_{k+1}}\theta^\ast(t)\dd t.
\end{align}

\begin{theorem}[Recovery rate for grid-sampled controls]
\label{prop:grid-sampled-recovery}
Assume Assumptions \ref{ass:controls}, \ref{ass:field}, \ref{ass:readout}, and \ref{ass:theta-lip}.  Let $\theta^\ast:[0,T]\to\Theta$ be Lipschitz with constant $L_{\theta^\ast}$.  Define the piecewise constant control $\thetaeps$ by either grid sampling or cell averaging as in \eqref{eq:continuous-to-discrete-control}.  Then there exists a constant $C_T$ independent of $\eps$ such that
\begin{align}
\label{eq:grid-sampled-recovery-rate}
    \abs{\mcJ(\thetaeps)-\mcJ(\theta^\ast)}\le C_T\eps.
\end{align}
Consequently, in the recovery estimates \eqref{eq:local-recovery-control} and in the global recovery condition, one may take $\alpha_\eps=O(\eps)$ whenever the sampled control belongs to the admissible finite-depth class $\calA_\eps$.
\end{theorem}

\begin{proof}
For the grid-sampled control, Lipschitz continuity of $\theta^\ast$ gives, for $t\in[t_k,t_{k+1})$,
\begin{align*}
    \norm{\theta^\ast_t-\thetaeps_t}
    \le L_{\theta^\ast}\eps.
\end{align*}
For the cell-average control the same $O(\eps)$ bound holds, possibly with a different constant.  Let $X_t$ and $\widetilde X_t$ be the controlled flows driven by $\theta^\ast$ and $\thetaeps$ from the same input.  Using Assumption \ref{ass:theta-lip},
\begin{align*}
    \frac{\dd}{\dd t}\norm{X_t-\widetilde X_t}
    \le \zeta_+L_F\norm{X_t-\widetilde X_t}
    +\zeta_+L_\theta\norm{\theta^\ast_t-\thetaeps_t}.
\end{align*}
Gronwall's inequality yields $\sup_{0\le t\le T}\norm{X_t-\widetilde X_t}\le C_T\eps$.  The Lipschitz bound \eqref{eq:ce-h-lip}, followed by integration with respect to $\mu_X$, gives \eqref{eq:grid-sampled-recovery-rate}.
\end{proof}

Therefore, \Cref{thm:local-minimizer-stability,thm:global-minimizer-stability} shows that local or global discrete minimizers initialized near this sampled control converge back to the corresponding continuous minimizer, up to the Euler, discretization, recovery, and optimization errors.

This continuous-to-discrete viewpoint also explains the role of the earlier approximation theorems.  The continuous problem identifies candidate depth profiles and layer-parameter curves through the transport equation \eqref{eq:transport} and the Pontryagin system \eqref{eq:pmp-state}--\eqref{eq:pmp-adjoint}, together with the stationarity condition \eqref{eq:pmp-stationarity}.  Once a candidate control has been found, the finite-depth Transformer inherits it by grid sampling or cell averaging as in \eqref{eq:continuous-to-discrete-control}.  The pathwise estimate \eqref{eq:pathwise-error} quantifies the state error introduced by replacing the continuous flow by finitely many residual layers.  The fixed-control estimate \eqref{eq:fixed-control-error}, the finite-class estimate \eqref{eq:finite-class-error}, and the entropy estimate \eqref{eq:metric-entropy-error} quantify the additional statistical error from using finitely many training samples.  Finally, \eqref{eq:empirical-to-full-value} combines the continuous-control approximation error, the Euler depth error, and the sampling error into a single bound for optimal values.

\section{Initialization and a priori range estimates from compact input support}
\label{sec:initialization-range}

We finally discuss how the compact input distribution can be used to choose an initialization and to estimate an a priori range for global minimizers.  The main point is that the input support alone does not determine a unique globally convergent initialization.  Rapid convergence of SGD requires an additional landscape condition, such as a PL inequality or strong convexity on a basin containing the initialization.  Compactness of $\operatorname{supp}\mu_X$ is nevertheless useful since it gives uniform bounds on states, logits, losses, and gradients, and therefore allows one to localize the search for minimizers.

Let $K_{\rm in}:=\operatorname{supp}\mu_X$ be compact.  Under the bounded-vector-field assumptions in \eqref{eq:bounded-vector-field-assumptions}, all trajectories generated by controls with values in a compact parameter set remain in a compact tube $K_T$ as in Proposition \ref{prop:uniform-local-compactness}.  Hence the loss and its control gradient are uniformly bounded on such a tube.  This justifies the following continuous-to-discrete initialization principle:
\begin{enumerate}
    \item solve a coarse or regularized version of the continuous control problem \eqref{eq:full-continuous-value};
    \item sample or average the resulting continuous control as in \eqref{eq:continuous-to-discrete-control};
    \item use the sampled sequence as the initial finite-depth Transformer parameters and then run projected SGD on the empirical finite-depth loss.  Here projected SGD means the constrained iteration
    \begin{align*}
        \theta^{n+1}=\Pi_{\calA_\eps}\left(\theta^n-\eta_n\nabla \widehat{\mcJ}_{N,\eps}(\theta^n)\right),
    \end{align*}
    where $\eta_n>0$ is a stepsize, $\widehat{\mcJ}_{N,\eps}$ denotes the empirical finite-depth objective \eqref{eq:empirical-risk} written in a Euclidean parameter chart, and $\Pi_{\calA_\eps}$ is the Euclidean projection onto the admissible discrete parameter set.  If $\calA_\eps=\Theta^M$ with compact convex $\Theta$, then $\Pi_{\calA_\eps}$ is simply the layerwise projection onto $\Theta$.
\end{enumerate}
This procedure initializes the discrete network near a continuous candidate whose objective value is controlled by \Cref{prop:grid-sampled-recovery}.

\begin{theorem}[Fast local SGD convergence from a continuous-control initialization]
\label{thm:sgd-initialization}
Let $\widehat{\mcJ}_{N,\eps}$ be the empirical finite-depth objective on a Euclidean parameter chart for $\calA_\eps$.  Suppose that there is a ball $B_R(\theta_\eps^\ast)$ around a global minimizer $\theta_\eps^\ast$ of $\widehat{\mcJ}_{N,\eps}$ such that the following hold on this ball:
\begin{enumerate}
    \item $\widehat{\mcJ}_{N,\eps}$ has $L_g$-Lipschitz gradient (see Proposition \ref{prop:empirical-risk-lg-smooth});
    \item $\widehat{\mcJ}_{N,\eps}$ satisfies the PL inequality (see Proposition \ref{prop:empirical-risk-pl-sufficient})
    \begin{align}
    \label{eq:pl-inequality}
        \frac12\norm{\nabla \widehat{\mcJ}_{N,\eps}(\theta)}^2
        \ge \mu_{\rm PL}\bigl(\widehat{\mcJ}_{N,\eps}(\theta)-\widehat{\mcJ}_{N,\eps}(\theta_\eps^\ast)\bigr);
    \end{align}
    \item Let
    \begin{align*}
        \mathcal G_n:=\operatorname{Alg}(\theta_0,g_0,\ldots,g_{n-1})
    \end{align*}
    be the training filtration generated by the past iterates and stochastic gradients. The stochastic gradient estimator $g_n$ is conditionally unbiased and has conditionally bounded variance:
    \begin{align}
    \label{eq:sgd-unbiased-variance}
        \E[g_n\mid\mathcal G_n]&=\nabla \widehat{\mcJ}_{N,\eps}(\theta_n),\\
        \E\bigl[\norm{g_n-\nabla \widehat{\mcJ}_{N,\eps}(\theta_n)}^2\mid\mathcal G_n\bigr]&\le \sigma_{\rm sg}^2.
    \end{align}
    In particular,
    \begin{align}
    \label{eq:sgd-second-moment}
        \E[\norm{g_n}^2\mid\mathcal G_n]
        \le \norm{\nabla \widehat{\mcJ}_{N,\eps}(\theta_n)}^2+\sigma_{\rm sg}^2.
    \end{align}
\end{enumerate}
If the initialization $\theta_0$ obtained by sampling a continuous candidate lies in $B_R(\theta_\eps^\ast)$ and the projected SGD iterates remain in this ball, then for any step size $0<\eta\le 1/L_g$,
\begin{align}
\label{eq:sgd-linear-rate}
    \E\bigl[\widehat{\mcJ}_{N,\eps}(\theta_n)-\widehat{\mcJ}_{N,\eps}(\theta_\eps^\ast)\bigr]
    \le (1-\eta\mu_{\rm PL})^n
    \bigl(\widehat{\mcJ}_{N,\eps}(\theta_0)-\widehat{\mcJ}_{N,\eps}(\theta_\eps^\ast)\bigr)
    +\frac{\eta L_g\sigma_{\rm sg}^2}{2\mu_{\rm PL}}.
\end{align}
In particular, if the continuous candidate is close enough that the initial objective gap is small, then SGD converges rapidly to a noise-dominated neighborhood of the global minimizer.
\end{theorem}

\begin{proof}
Let $\E_n[\cdot]=\E[\cdot\mid\mathcal G_n]$.  We first write the argument for the unconstrained stochastic step $\theta_{n+1}=\theta_n-\eta g_n$.  
The descent lemma for an $L_g$-smooth function states that
\begin{align*}
    f(y)\le f(x)+\inner{\nabla f(x)}{y-x}+\frac{L_g}{2}\norm{y-x}^2,
\end{align*}
for all $x,y$ in the smoothness region; see, for example, \citet[Section~2.1]{Nesterov2018}.  Applying this with $f=\widehat{\mcJ}_{N,\eps}$, $x=\theta_n$, and $y=\theta_n-\eta g_n$, and then taking conditional expectation, gives
\begin{align*}
    \E_n\widehat{\mcJ}_{N,\eps}(\theta_{n+1})
    &\le \widehat{\mcJ}_{N,\eps}(\theta_n)
    -\eta\inner{\nabla\widehat{\mcJ}_{N,\eps}(\theta_n)}{\E_n g_n}
    +\frac{L_g\eta^2}{2}\E_n\norm{g_n}^2.
\end{align*}
Using the unbiasedness and second-moment estimates \eqref{eq:sgd-unbiased-variance}--\eqref{eq:sgd-second-moment}, we obtain
\begin{align*}
    \E_n\widehat{\mcJ}_{N,\eps}(\theta_{n+1})
    &\le \widehat{\mcJ}_{N,\eps}(\theta_n)
    -\eta\norm{\nabla\widehat{\mcJ}_{N,\eps}(\theta_n)}^2
    +\frac{L_g\eta^2}{2}\left(\norm{\nabla\widehat{\mcJ}_{N,\eps}(\theta_n)}^2+\sigma_{\rm sg}^2\right)\\
    &=\widehat{\mcJ}_{N,\eps}(\theta_n)
    -\eta\left(1-\frac{L_g\eta}{2}\right)\norm{\nabla\widehat{\mcJ}_{N,\eps}(\theta_n)}^2
    +\frac{L_g\eta^2}{2}\sigma_{\rm sg}^2.
\end{align*}
Since $\eta\le1/L_g$, the coefficient satisfies $1-L_g\eta/2\ge1/2$.  Applying the PL inequality \eqref{eq:pl-inequality} gives
\begin{align*}
    \E_n\Delta_{n+1}
    \le (1-\eta\mu_{\rm PL})\Delta_n+\frac{L_g\eta^2}{2}\sigma_{\rm sg}^2,
\end{align*}
where $\Delta_n=\widehat{\mcJ}_{N,\eps}(\theta_n)-\widehat{\mcJ}_{N,\eps}(\theta_\eps^\ast)$.  Taking total expectation gives the affine recursion
\begin{align*}
    \E\Delta_{n+1}
    \le q\,\E\Delta_n+b,
    \qquad
    q:=1-\eta\mu_{\rm PL},
    \qquad
    b:=\frac{L_g\eta^2}{2}\sigma_{\rm sg}^2.
\end{align*}
Iterating yields
\begin{align*}
    \E\Delta_n
    \le q^n\Delta_0+b\sum_{m=0}^{n-1}q^m
    \le q^n\Delta_0+\frac{b}{1-q}.
\end{align*}
Since $1-q=\eta\mu_{\rm PL}$, the steady-state term is
\begin{align*}
    \frac{b}{1-q}
    =\frac{(L_g\eta^2/2)\sigma_{\rm sg}^2}{\eta\mu_{\rm PL}}
    =\frac{\eta L_g\sigma_{\rm sg}^2}{2\mu_{\rm PL}}.
\end{align*}
This proves \eqref{eq:sgd-linear-rate}.
\end{proof}

\begin{remark}[How to choose the initialization]
\Cref{thm:sgd-initialization} suggests the following practical initialization rule.  First compute a low-resolution continuous or coarse-depth optimizer for \eqref{eq:full-continuous-value}, preferably with the regularized objective $\mcJ_\gamma$ when one wants to enforce a stable basin.  Second, convert the continuous control to a finite-depth layer sequence by grid sampling or cell averaging as in \eqref{eq:continuous-to-discrete-control}.  The deterministic recovery estimate \eqref{eq:grid-sampled-recovery-rate} in \Cref{prop:grid-sampled-recovery}, together with the value comparison \eqref{eq:empirical-to-full-value} in \Cref{thm:full-control-approximation}, controls the objective gap introduced by this continuous-to-discrete conversion.  Third, one may add a small perturbation whose norm is below the radius of the PL basin $B_R(\theta_\eps^\ast)$.  Once the initialization lies in that basin, the PL inequality \eqref{eq:pl-inequality} converts the small initial gap into the linear, or geometric, convergence estimate \eqref{eq:sgd-linear-rate}, up to the noise floor $\eta L_g\sigma_{\rm sg}^2/(2\mu_{\rm PL})$. 
\end{remark}

Without a PL or strong-convexity basin, no initialization rule depending only on $\mu_X$ can guarantee rapid convergence for a nonconvex Transformer objective.  This obstruction can already be realized inside the simplified Transformer setting of this paper.  Take
\begin{align*}
    d_x=L=d_v=d_k=2,
    \qquad N=1,
    \qquad \mu_X=\delta_{X_0},
    \qquad Y_0(X_0)=e_1,
\end{align*}
freeze all parameters except a scalar parameter $a$ in a smooth one-dimensional subfamily $\theta(a)$, and consider one residual layer, so that
\begin{align*}
    X_1^\eps(a)=X_0+\eps\zeta(0)F(X_0;\theta(a)).
\end{align*}
By freezing the attention part and using one feed-forward coordinate, or equivalently one value-output coordinate, one can realize an explicit one-dimensional reduced landscape.  Let
\begin{align*}
    e_{\rm gap}:=e_1-e_2\in\R^{d_v},
    \qquad
    s(a):=H(X_1^\eps(a))_{0,1}-H(X_1^\eps(a))_{0,2}
    =\inner{H(X_1^\eps(a))_{0,:}}{e_{\rm gap}}.
\end{align*}
Choose a vector $q\in\R^{d_x}$ with $\norm q=1$ and write $x_{0,0}:=(X_0)_{0,:}\in\R^{d_x}$ for the first token of the fixed input.  In this toy example we keep the readout in the linear form used earlier, namely
\begin{align*}
    H(X)=XW_{\mathrm{out}},
    \qquad
    W_{\mathrm{out}}=\left(\frac{q}{2},-\frac{q}{2}\right)\in\R^{d_x\times 2}.
\end{align*}
Thus, at the fixed initial input $X_0$, for the first token and the two vocabulary coordinates,
\begin{align*}
    H(X_0)_{0,1}=\frac12\inner{q}{x_{0,0}},
    \qquad
    H(X_0)_{0,2}=-\frac12\inner{q}{x_{0,0}}.
\end{align*}
Therefore the initial logit gap is
\begin{align*}
    \inner{H(X_0)_{0,:}}{e_{\rm gap}}
    =H(X_0)_{0,1}-H(X_0)_{0,2}
    =\inner{q}{x_{0,0}}.
\end{align*}
Since
\begin{align*}
    (X_{1}^{\eps})_{0,:}(a)=x_{0,0}+\eps\zeta(0)F(X_0;\theta(a))_{0,:},
\end{align*}
we obtain
\begin{align*}
    \inner{H(X_1^\eps(a))_{0,:}}{e_{\rm gap}}
    =\inner{q}{x_{0,0}}+
      \eps\zeta(0)\inner{q}{F(X_0;\theta(a))_{0,:}}.
\end{align*}
Thus, assuming for this toy one-layer example that $\zeta(0)>0$, it is enough to prescribe the selected scalar vector-field response as
\begin{align*}
    \inner{q}{F(X_0;\theta_\pm(a))_{0,:}}
    =\frac{S-\inner{q}{x_{0,0}}-\kappa(a^2-1)^2\mp\tau a+R_\eta(a)}{\eps\zeta(0)},
    \qquad
    \sup_{a\in I}\abs{R_\eta(a)}\le\eta.
\end{align*}
The preceding scalar prescription is compatible with the model vector field \eqref{eq:transformer-vector-field}.  One explicit realization is obtained by setting the attention contribution and the non-bias feed-forward contribution to zero at the chosen input, for instance by taking $W_V=0$, $W_1=0$, $W_2=0$, and choosing a smooth one-dimensional bias curve
\begin{align*}
    b_{2,\pm}(a)
    =\frac{S-\inner{q}{x_{0,0}}-\kappa(a^2-1)^2\mp\tau a+R_\eta(a)}{\eps\zeta(0)}\,q.
\end{align*}
Then
\begin{align*}
    F(X_0;\theta_\pm(a))
    =\mathbf 1_L b_{2,\pm}(a)^\top,
    \qquad
    \inner{q}{F(X_0;\theta_\pm(a))_{0,:}}
    =\frac{S-\inner{q}{x_{0,0}}-\kappa(a^2-1)^2\mp\tau a+R_\eta(a)}{\eps\zeta(0)}.
\end{align*}
Equivalently, one may freeze the feed-forward block and realize the same selected scalar response through a value-output direction: choose fixed attention weights at $X_0$ and a smooth curve in $W_V(a)W_O(a)$ so that
\begin{align*}
    \inner{q}{\bigl[A_{\theta_\pm(a)}(X_0)X_0W_V(a)W_O(a)\bigr]_{0,:}}
    =\frac{S-\inner{q}{x_{0,0}}-\kappa(a^2-1)^2\mp\tau a+R_\eta(a)}{\eps\zeta(0)}.
\end{align*}
Both constructions use only a smooth one-dimensional subfamily $a\mapsto\theta_\pm(a)$ inside the finite-dimensional Transformer parameter space; hence the double-well behavior is already present in a very low-dimensional restriction of the admissible model class.
With this choice,
\begin{align}
\label{eq:toy-realized-logit-gap}
    \inner{H(X_1^\eps(a))_{0,:}}{e_{\rm gap}}
    = \phi_\pm(a)+R_\eta(a),
\end{align}
where
\begin{align}
\label{eq:toy-double-well-logit-gap}
    \phi_\pm(a)=S-\kappa(a^2-1)^2\mp \tau a,
    \qquad
    S\gg1,
    \quad \kappa>0,
    \quad 0<\tau\ll\kappa.
\end{align}
Equivalently,
\begin{align*}
    s_\pm(a)
    :=H(X_1^\eps(a))_{0,1}-H(X_1^\eps(a))_{0,2}
    =S-\kappa(a^2-1)^2\mp\tau a+\mathcal O(\eta)
\end{align*}
uniformly for $a\in I$. 
For binary cross-entropy with the correct label $e_1$, only the logit gap matters:
\begin{align*}
    \widehat{\mcJ}_{N,\eps}^{\pm}(\theta(a))
    &= -\log\frac{e^{H(X_1^\eps(a))_{0,1}}}
        {e^{H(X_1^\eps(a))_{0,1}}+e^{H(X_1^\eps(a))_{0,2}}}       \\
    &=\log\bigl(1+e^{-s_\pm(a)}\bigr).
\end{align*}
Set $f(u)=\log(1+e^{-u})$ and write
\begin{align*}
    s_\pm(a)=S+\Delta_\pm(a),
    \qquad
    \Delta_\pm(a)=-\kappa(a^2-1)^2\mp\tau a+\mathcal O(\eta).
\end{align*}
Since
\begin{align*}
    f'(u)=-\frac{1}{1+e^u},
    \qquad
    f''(u)=\frac{e^u}{(1+e^u)^2},
\end{align*}
Taylor expansion at $S$ gives, uniformly on $I$,
\begin{align*}
    f(S+\Delta_\pm(a))
    &=f(S)+f'(S)\Delta_\pm(a)
      +\mathcal O\left(\sup_{u\in S+\Delta_\pm(I)} f''(u)\,\Delta_\pm(a)^2\right)  \\
    &=C_0+c(a^2-1)^2\pm\tau_1 a
      +\mathcal O\left(e^{-S}\bigl((a^2-1)^4+\tau^2a^2+\eta^2\bigr)+e^{-S}\eta\right),
\end{align*}
where
\begin{align*}
    C_0=f(S),
    \qquad
    c=\frac{\kappa}{1+e^S}\simeq e^{-S}\kappa>0,
    \qquad
    \tau_1=\frac{\tau}{1+e^S}\simeq e^{-S}\tau>0.
\end{align*}
Thus the two reduced losses are small perturbations of
\begin{align*}
    C_0+c(a^2-1)^2+\tau_1 a
    \qquad\text{and}\qquad
    C_0+c(a^2-1)^2-\tau_1 a.
\end{align*}
Their preferred wells lie near opposite signs of $a$: the first objective favors the well near $a=-1$, while the second favors the well near $a=1$.  Hence the same input law $\mu_X=\delta_{X_0}$ can generate two admissible one-dimensional Transformer landscapes whose good basins are separated. This is a concrete finite-dimensional version of the general nonconvex landscape phenomenon for neural-network training objectives; see, for example, \citet{GoodfellowVinyalsSaxe2015} and \citet{ChoromanskaHenaffMathieuArousLeCun2015}.  Any initialization rule depending only on $\mu_X$ assigns the same initial value in both landscapes, and therefore cannot guarantee rapid convergence to the correct basin in both problems without additional information such as a PL basin, strong convexity, or a landscape-dependent initialization.

The initialization theorem below assumes a local PL basin and a smooth empirical objective.  We first give a checkable smoothness statement for the finite-dimensional layer-parameter chart.
\begin{proposition}[Local Lipschitz gradient of the empirical finite-depth risk]
\label{prop:empirical-risk-lg-smooth}
Fix $N<\infty$ and $\eps>0$, and identify a finite-depth layer sequence with a Euclidean parameter vector $\theta\in\R^{d_{\theta}}$.  Let $B_R(\theta_\eps^\ast)$ be a ball on which the discrete trajectories remain in a compact state set and the layer parameters remain in a compact parameter set.  Assume that the Transformer vector field $F$ and the readout $H$ are twice continuously differentiable in all state and parameter variables on these compact sets, with all derivatives up to order two bounded there.  Then the empirical finite-depth objective
\begin{align*}
    \widehat{\mcJ}_{N,\eps}(\theta)
    =\frac1N\sum_{i=1}^N \ell\bigl(Y^i,H(X_M^{\eps,i}(\theta))\bigr)
\end{align*}
has an $L_g$-Lipschitz gradient on $B_R(\theta_\eps^\ast)$; that is,
\begin{align}
\label{eq:empirical-risk-lg-smooth}
    \norm{\nabla\widehat{\mcJ}_{N,\eps}(\theta)-\nabla\widehat{\mcJ}_{N,\eps}(\vartheta)}
    \le L_g\norm{\theta-\vartheta},
    \qquad \theta,\vartheta\in B_R(\theta_\eps^\ast).
\end{align}
The constant $L_g$ may depend on $T$, $\eps$, $N$ only through the fixed finite-depth chart, the compact state and parameter bounds, and the derivative bounds of $F$, $H$, and the cross-entropy loss on the corresponding bounded logit range.
\end{proposition}

\begin{proof}
For a fixed sample $i$, write $X_k^i(\theta)$ for the discrete trajectory.  The recursion
\begin{align*}
    X_{k+1}^i(\theta)=X_k^i(\theta)+\eps\zeta(t_k)F(X_k^i(\theta);\theta_k)
\end{align*}
is a finite composition of $C^2$ maps on compact sets.  Therefore the parameter-to-state map $\theta\mapsto X_M^i(\theta)$ is $C^2$ on $B_R(\theta_\eps^\ast)$.  More explicitly, differentiating the recursion once gives a linear sensitivity recursion of the form
\begin{align*}
    D_\theta X_{k+1}^i
    &=D_\theta X_k^i
      +\eps\zeta(t_k)D_XF(X_k^i;\theta_k)D_\theta X_k^i
      +\eps\zeta(t_k)D_\theta F(X_k^i;\theta_k)\Pi_k,
\end{align*}
where $\Pi_k$ extracts the $k$-th layer parameter from the full layer vector.  Since $D_XF$ and $D_\theta F$ are bounded on the compact set, a discrete Gronwall estimate gives
\begin{align*}
    \sup_{\theta\in B_R(\theta_\eps^\ast)}\norm{D_\theta X_M^i(\theta)}\le C_1.
\end{align*}
Differentiating the same recursion a second time gives an affine recursion for $D_\theta^2X_k^i$ whose coefficients contain only $D_XF$, $D_\theta F$, $D_{XX}^2F$, $D_{X\theta}^2F$, and $D_{\theta\theta}^2F$, all bounded on the compact set.  Applying discrete Gronwall again yields
\begin{align*}
    \sup_{\theta\in B_R(\theta_\eps^\ast)}\norm{D_\theta^2 X_M^i(\theta)}\le C_2.
\end{align*}
The sample loss
\begin{align*}
    f_i(\theta)=\ell\bigl(Y^i,H(X_M^i(\theta))\bigr)
\end{align*}
is also $C^2$.  The softmax cross-entropy has bounded first and second derivatives on bounded logit sets, and $H$ has bounded first and second derivatives on the compact state set.  By the chain rule,
\begin{align*}
    \norm{\nabla^2 f_i(\theta)}
    \le C\left(1+\norm{D_\theta X_M^i(\theta)}^2+\norm{D_\theta^2X_M^i(\theta)}\right)
    \le C_3,
\end{align*}
for all $\theta\in B_R(\theta_\eps^\ast)$, with $C_3$ independent of $i$.  Since
\begin{align*}
    \nabla^2\widehat{\mcJ}_{N,\eps}(\theta)=\frac1N\sum_{i=1}^N\nabla^2 f_i(\theta),
\end{align*}
we have $\sup_{B_R}\norm{\nabla^2\widehat{\mcJ}_{N,\eps}}\le C_3$.  The mean-value theorem then gives \eqref{eq:empirical-risk-lg-smooth} with $L_g=C_3$.
\end{proof}

Smoothness alone does not imply a PL inequality.  The next proposition explains how local Hessian coercivity on identifiable directions yields the PL condition used in the SGD basin argument.
\begin{proposition}[A checkable sufficient condition for the local PL inequality]
\label{prop:empirical-risk-pl-sufficient}
Let $\theta_\eps^\ast$ be a local minimizer of $\widehat{\mcJ}_{N,\eps}$ in a Euclidean finite-depth parameter chart.  Assume that the local smoothness conclusion of Proposition \ref{prop:empirical-risk-lg-smooth} holds on $B_R(\theta_\eps^\ast)$.  Let $\mathcal T_{\rm id}$ denote the tangent space of the identifiable local chart obtained after quotienting out exact flat symmetry directions.\footnote{Here an exact flat symmetry direction means a parameter variation that leaves the realized input-output map, and therefore the empirical objective, unchanged in a neighborhood.}  Assume that, for some $\lambda>0$,
\begin{align}
\label{eq:local-hessian-coercive-empirical}
    v^\top \nabla^2\widehat{\mcJ}_{N,\eps}(\theta)v
    \ge \lambda \norm{v}^2,
    \qquad
    \theta\in B_R(\theta_\eps^\ast),\quad v\in \mathcal T_{\rm id}.
\end{align}
Then, on the same identifiable local chart, $\widehat{\mcJ}_{N,\eps}$ satisfies the PL inequality
\begin{align}
\label{eq:empirical-risk-local-pl-derived}
    \frac12\norm{\nabla\widehat{\mcJ}_{N,\eps}(\theta)}^2
    \ge \mu_{\rm PL}\bigl(\widehat{\mcJ}_{N,\eps}(\theta)-\widehat{\mcJ}_{N,\eps}(\theta_\eps^\ast)\bigr),
    \qquad
    \mu_{\rm PL}=\frac{\lambda^2}{L_g},
\end{align}
for all $\theta\in B_R(\theta_\eps^\ast)$ for which the segment between $\theta$ and $\theta_\eps^\ast$ stays in the chart.  In particular, \eqref{eq:pl-inequality} holds in any nondegenerate local basin where the empirical Hessian is uniformly positive on the identifiable directions.
\end{proposition}

\begin{proof}
Let $\delta=\theta-\theta_\eps^\ast\in\mathcal T_{\rm id}$. Since $\theta_\eps^\ast$ is a local minimizer in the identifiable chart, $\nabla\widehat{\mcJ}_{N,\eps}(\theta_\eps^\ast)=0$ on the identifiable tangent directions.  By the fundamental theorem of calculus and \eqref{eq:local-hessian-coercive-empirical},
\begin{align*}
    \inner{\nabla\widehat{\mcJ}_{N,\eps}(\theta)}{\delta}
    &=\int_0^1 \delta^\top
      \nabla^2\widehat{\mcJ}_{N,\eps}(\theta_\eps^\ast+s\delta)\delta\,\dd s
      \ge \lambda\norm{\delta}^2.
\end{align*}
Hence Cauchy's inequality gives
\begin{align}
\label{eq:gradient-distance-lower-bound}
    \norm{\nabla\widehat{\mcJ}_{N,\eps}(\theta)}\ge \lambda\norm{\theta-\theta_\eps^\ast}.
\end{align}
On the other hand, the $L_g$-Lipschitz gradient bound implies the standard upper quadratic estimate around the stationary point,
\begin{align}
\label{eq:objective-gap-upper-smoothness}
    \widehat{\mcJ}_{N,\eps}(\theta)-\widehat{\mcJ}_{N,\eps}(\theta_\eps^\ast)
    \le \frac{L_g}{2}\norm{\theta-\theta_\eps^\ast}^2.
\end{align}
Combining \eqref{eq:gradient-distance-lower-bound} and \eqref{eq:objective-gap-upper-smoothness} yields
\begin{align*}
    \frac12\norm{\nabla\widehat{\mcJ}_{N,\eps}(\theta)}^2
    \ge \frac{\lambda^2}{2}\norm{\theta-\theta_\eps^\ast}^2
    \ge \frac{\lambda^2}{L_g}
      \bigl(\widehat{\mcJ}_{N,\eps}(\theta)-\widehat{\mcJ}_{N,\eps}(\theta_\eps^\ast)\bigr),
\end{align*}
which is \eqref{eq:empirical-risk-local-pl-derived}.  The tangent-space qualification is necessary in the Transformer setting since parameter symmetries, redundant heads, and factorized value-output directions can create exact null directions.  The PL conclusion is therefore a local nondegeneracy condition on the identifiable directions, not a global convexity statement.
\end{proof}

We finally record range estimates that justify restricting the search for regularized minimizers to bounded parameter sets.  These estimates complement the earlier compactness assumptions by giving one way to obtain them from the objective itself.
\begin{theorem}[A priori range estimate for regularized global minimizers]
\label{thm:minimizer-range}
Assume $K_{\rm in}=\operatorname{supp}\mu_X$ is compact and consider the regularized continuous objective
\begin{align*}
    \mcJ_\gamma(\theta)=\mcJ(\theta)+\frac{\gamma}{2}\int_0^T\norm{\theta_t}^2\dd t,
    \qquad \gamma>0.
\end{align*}
Let $\theta^\ast$ be a global minimizer of $\mcJ_\gamma$ in a class containing the zero control.  Then
\begin{align}
\label{eq:l2-range-control}
    \int_0^T\norm{\theta_t^\ast}^2\dd t
    \le \frac{2}{\gamma}\mcJ_\gamma(0).
\end{align}
If, in addition, admissible controls are $L_{\calU}$-Lipschitz in time, then every global minimizer satisfies the pointwise range estimate
\begin{align}
\label{eq:pointwise-range-control}
    \sup_{0\le t\le T}\norm{\theta_t^\ast}
    \le \sqrt{\frac{2\mcJ_\gamma(0)}{\gamma T}}+L_{\calU}T.
\end{align}
\end{theorem}

\begin{proof}
Since $\theta^\ast$ minimizes $\mcJ_\gamma$ and the loss is nonnegative,
\begin{align*}
    \frac{\gamma}{2}\int_0^T\norm{\theta_t^\ast}^2\dd t
    \le \mcJ_\gamma(\theta^\ast)
    \le \mcJ_\gamma(0),
\end{align*}
which proves \eqref{eq:l2-range-control}. Let
\begin{align*}
    A:=T^{-1}\int_0^T\norm{\theta_t^\ast}^2\dd t.
\end{align*}
There must exist some $s\in[0,T]$ with $\norm{\theta_s^\ast}^2\le A$; otherwise $\norm{\theta_t^\ast}^2>A$ for every $t$, and integrating would give $\int_0^T\norm{\theta_t^\ast}^2\dd t>TA$, a contradiction.  Hence
\begin{align*}
    \norm{\theta_s^\ast}\le T^{-1/2}\left(\int_0^T\norm{\theta_t^\ast}^2\dd t\right)^{1/2}.
\end{align*}  For any $t\in[0,T]$,
\begin{align*}
    \norm{\theta_t^\ast}
    \le \norm{\theta_s^\ast}+L_{\calU}\abs{t-s}
    \le \sqrt{\frac{2\mcJ_\gamma(0)}{\gamma T}}+L_{\calU}T.
\end{align*}
This proves \eqref{eq:pointwise-range-control}.
\end{proof}

\begin{remark}[Estimating the range from $\mu_X$]
The quantity $\mcJ_\gamma(0)$ is computable or estimable from the input distribution.  Under the zero control, the terminal state is determined by the flow with $\theta_t\equiv0$; if the corresponding logits are bounded on the compact input support, then
\begin{align*}
    \mcJ_\gamma(0)=\E_{X_0\sim\mu_X}\ell(Y(X_0),H(X_T^0(X_0)))
\end{align*}
where $X_T^0(X_0)$ denotes the terminal state generated by the zero control $\theta_t\equiv0$.
If the parameter origin is chosen so that $F(X;0)=0$ on the relevant state set, then $X_T^0(X_0)=X_0$. 
This quantity can be estimated by Monte Carlo samples from $\mu_X$.  The compact support of $\mu_X$ and the boundedness assumptions on $F$ and $H$ ensure this estimate is finite and uniform.  Therefore \eqref{eq:pointwise-range-control} gives an explicit data-dependent radius for the parameter search region.  Without compactness, regularization, or a prescribed admissible compact set $\Theta$, such a range estimate is generally impossible for a noncoercive overparameterized Transformer objective.  For example, take the toy dimensions $d_x=L=d_v=d_k=2$.  Fix compactly supported inputs and consider the one-parameter family in which all parameters except the output matrix $W_O$ are fixed, while the value matrix is set to $W_V=0$.  Then
\begin{align*}
    V=XW_V=0,
    \qquad
    A_\theta(X)VW_O=A_\theta(X)0\,W_O=0
\end{align*}
for every $X$ and for every $W_O\in\R^{2\times2}$.  If the feed-forward part and the readout are fixed so that the terminal loss is minimized independently of this unused attention-output direction, then
\begin{align*}
    \mcJ(\theta(W_O))=\mcJ(\theta(0))
    \qquad\text{for all }W_O\in\R^{2\times2}.
\end{align*}
Consequently, for every radius $R>0$ one can choose $W_O$ with $\norm{W_O}>R$ and obtain the same objective value.  This gives an unbounded family of minimizers with the same input law and the same loss value.  A finite range estimate cannot be derived from $\mu_X$ alone; one must add a compact parameter constraint, a coercive regularizer such as $\gamma\int_0^T\norm{\theta_t}^2\dd t$, or an identifiability condition that removes unused parameter directions.
\end{remark}

The Pontryagin condition gives a sharper pointwise estimate under additional structure adapted to the Transformer vector field.  Suppose that the regularized optimum is an interior point of the admissible set, that the regularization is added as $\frac{\gamma}{2}\int_0^T\norm{\theta_t}^2\dd t$, and that, on the compact tube $K_T$, the Transformer derivatives satisfy
\begin{align}
\label{eq:transformer-dtheta-bound}
    \norm{D_\theta F(X;\theta)^\top P}\le C_{F,\theta}\norm{P}
    \qquad X\in K_T,
\end{align}
for all relevant $\theta$.  For the simplified attention/FFN field \eqref{eq:transformer-vector-field}, such a bound follows from boundedness of the hidden states, boundedness of the parameter set, smoothness of the masked softmax on finite rows, and bounded derivatives of $\sigma$.  The regularized stationarity condition becomes
\begin{align}
\label{eq:regularized-pmp-stationarity}
    \gamma\theta_t^\ast+\E\left[D_\theta F(X_t^\ast;\theta_t^\ast)^\top P_t^\ast\right]=0.
\end{align}
If the adjoint is uniformly bounded (see Proposition \ref{prop:uniform-adjoint-bound}), $\norm{P_t^\ast}\le B_P$, then \eqref{eq:transformer-dtheta-bound}--\eqref{eq:regularized-pmp-stationarity} imply the pointwise estimate
\begin{align}
\label{eq:pmp-pointwise-range}
    \norm{\theta_t^\ast}\le \frac{C_{F,\theta}B_P}{\gamma}
    \qquad \text{for a.e. }t\in[0,T].
\end{align}
This bound can be sharper than \eqref{eq:pointwise-range-control} since it uses the optimality system rather than converting an $L^2$-bound into an $L^\infty$-bound through a Lipschitz-in-time assumption.
Thus the search for a global minimizer can be restricted to the compact parameter ball with radius given by the right-hand side of \eqref{eq:pointwise-range-control}, intersected with the original parameter constraint set.

The sharper pointwise range estimate used above relies on a uniform bound for the adjoint variable.  The following proposition derives such a bound directly from the terminal condition and Gronwall's inequality.
\begin{proposition}[Uniform adjoint bound]
\label{prop:uniform-adjoint-bound}
Assume that $H$ is continuously differentiable on the compact tube $K_T$ and that $D_XF$ is uniformly bounded there: there are constants $B_{DH},B_{DXF}<\infty$ such that
\begin{align*}
    \norm{DH(X)}\le B_{DH},
    \qquad
    \norm{D_XF(X;\theta)}\le B_{DXF},
    \qquad X\in K_T,
\end{align*}
for all relevant $\theta$.  Then every adjoint trajectory solving \eqref{eq:pmp-adjoint} satisfies
\begin{align}
\label{eq:uniform-adjoint-bound}
    \norm{P_t^\ast}\le B_P
    :=\frac{2B_{DH}}{\sqrt L}\exp(\zeta_+B_{DXF}T),
    \qquad 0\le t\le T.
\end{align}
\end{proposition}

\begin{proof}
By \eqref{eq:ce-gradient}, $\norm{\nabla_Z\ell(Y,Z)}\le 2/\sqrt L$ for all $Y\in\DeltaV^L$.  Hence the terminal condition in \eqref{eq:pmp-adjoint} gives
\begin{align*}
    \norm{P_T^\ast}
    \le \norm{DH(X_T^\ast)}\,\norm{\nabla_Z\ell(Y(X_0),H(X_T^\ast))}
    \le \frac{2B_{DH}}{\sqrt L}.
\end{align*}
The adjoint equation implies, for $t\le s\le T$,
\begin{align*}
    \frac{\dd}{\dd s}\norm{P_s^\ast}
    \le \zeta(s)\norm{D_XF(X_s^\ast;\theta_s^\ast)}\norm{P_s^\ast}
    \le \zeta_+B_{DXF}\norm{P_s^\ast}.
\end{align*}
Applying Gronwall backward from $T$ to $t$ yields
\begin{align*}
    \norm{P_t^\ast}
    \le \norm{P_T^\ast}\exp(\zeta_+B_{DXF}(T-t))
    \le \frac{2B_{DH}}{\sqrt L}\exp(\zeta_+B_{DXF}T),
\end{align*}
which proves \eqref{eq:uniform-adjoint-bound}.
\end{proof}

\section{Discussion, extensions, and conclusion}
\label{sec:discussion-conclusion}

\subsection{What has been proved}

The main purpose of this paper is to separate, in a mathematically explicit way, three sources of error that coexist in finite Transformer training.  The first source is depth discretization: the residual recursion \eqref{eq:discrete-transformer} is an explicit Euler approximation of the continuous controlled flow \eqref{eq:continuous-transformer}, leading to the pathwise estimate \eqref{eq:pathwise-error} and the population discretization bound \eqref{eq:pop-disc-error}.  The second source is statistical sampling: the fixed-control concentration estimate \eqref{eq:fixed-control-error}, the finite-class estimate \eqref{eq:finite-class-error}, and the metric-entropy estimate \eqref{eq:metric-entropy-error} quantify how empirical risk approximates population risk. Among these, the fixed-control estimate has the most direct interpretation.  The finite-class and metric-entropy bounds are uniform over fixed admissible classes and should not be read as a complete generalization theory for a Transformer chosen adaptively by SGD from the same sample. The third source is control-class approximation: the comparison bounds \eqref{eq:value-comparison}, \eqref{eq:approx-minimizer}, and \eqref{eq:empirical-to-full-value} explain how finite-depth empirical optima relate to continuous-depth population optima.  These estimates show that the continuous model is not merely a formal limit; it gives a quantitative baseline for comparing finite residual Transformers with a population control problem.

The limiting object is a first-order mean field control problem for the law of hidden states along depth.  In this formulation, the control is the layer-parameter curve, the state is the population distribution $\rho_t^\theta$, and the terminal loss is the token-level cross-entropy evaluated along characteristics. This is different from mean-field neural-network limits in which the empirical measure is taken over neurons, heads, or parameters; here the transported measure is the data-induced hidden-state law under a common control. The Pontryagin condition in \Cref{prop:pmp} identifies the cross-entropy softmax residual as the terminal adjoint signal.  This is the continuous-depth counterpart of the backpropagated classification error in a finite Transformer.

\subsection{Meaning of the assumptions}

The assumptions are deliberately stated in compact, local form.  They are not intended to say that every practical Transformer trajectory is automatically compact or globally smooth.  Rather, they identify what is needed for the Euler, concentration, and control-stability arguments to be valid.  Compactness can be enforced by projection, obtained conditionally through a priori range estimates, or motivated by normalized architectures.  In particular, layer normalization maps each token to a bounded subset of the mean-zero subspace, as discussed in \Cref{sec:assumptions}.  Regularization and compact control classes provide another route to range control, as shown in \Cref{thm:minimizer-range}.

The simplified vector field \eqref{eq:transformer-vector-field} should be read as a representative residual block rather than as the only architecture covered by the method.  Multi-head attention, feed-forward blocks, smooth layer-normalization approximations, trainable readouts, and residual gates can be handled by the same proof strategy whenever the resulting block map satisfies the boundedness and Lipschitz estimates used in Assumptions \ref{ass:controls}--\ref{ass:theta-lip}.  For example, in a multi-head attention architecture, if
\begin{align*}
    \widetilde F(X;\widetilde\theta)=\sum_{m=1}^{M_0}G_m(X;\widetilde\theta_m)
\end{align*}
and each component $G_m$ is bounded and Lipschitz on the relevant compact set, then $\widetilde F$ is also bounded and Lipschitz, with constants bounded by the sums of the component constants.  For finite compositions of smooth components, the same conclusion follows from the chain rule and compactness.  Thus the present analysis extends to many finite architectural modifications after changing constants, although not to singular or unbounded regimes without additional estimates.

\subsection{Optimization and nonconvexity}

The paper does not claim global convexity of Transformer training.  On the contrary, the local stability and initialization results are formulated to make nonconvexity explicit.  The quadratic-growth and PL assumptions are local basin conditions. They are not derived from the Transformer architecture or from cross-entropy alone; Example \ref{ex:frozen-feature-readout-qg} gives one restricted frozen-feature readout setting where such assumptions can be checked, while the general Transformer case requires separate verification on the basin of interest. The toy construction in \Cref{sec:initialization-range} shows that even a one-dimensional subfamily of the simplified Transformer model can produce separated nonconvex basins under cross-entropy loss.  Therefore the local SGD estimate \eqref{eq:sgd-linear-rate} should be interpreted as a basin-stability result: once an initialization lies in a suitable PL neighborhood, stochastic gradient descent decreases the objective gap geometrically up to a noise floor.  It is not a theorem guaranteeing global discovery of that neighborhood.

A global initialization statement would require an additional basin-discovery mechanism.  One simple formulation is probabilistic.  Let $B_R(\theta_\eps^\ast)$ be a local PL basin and let $\nu_0$ be a distribution over initial layer sequences.  If
\begin{align*}
    p_R:=\nu_0\bigl(B_R(\theta_\eps^\ast)\bigr)>0,
\end{align*}
then $K$ independent restarts hit the basin with probability
\begin{align*}
    1-(1-p_R)^K.
\end{align*}
Conditioned on this event, \eqref{eq:sgd-linear-rate} gives the local convergence rate.  This separation between global basin discovery and local basin convergence is important for interpreting continuous-control initializations: a continuous solve, grid sampling, or coarse-to-fine procedure may provide useful initial candidates, but a separate landscape argument is needed to ensure that such candidates enter a good basin uniformly.

\subsection{Extensions}

Several extensions are natural.  First, the statistical estimates in this paper are intentionally elementary.  The finite-class and metric-entropy bounds can likely be sharpened for structured Transformer control classes using temporal regularity, parameter sharing, sparsity, PAC-Bayesian arguments, or data-dependent complexity.  A more practical generalization theory for trained Transformers would also need to account for the data-dependent choice of the final parameters, for example through algorithmic stability of SGD, compression, localized Rademacher complexity, or another complexity measure adapted to the trajectory of training. Such refinements would affect the sampling term but would not change the Euler part of the analysis.

Second, the present state equation uses a single residual stream with a scalar depth profile $\zeta(t)$.  This captures a residual strength but not the full structure of hyper-connections, multiple streams, or manifold-constrained mixing.  A more faithful continuous-depth model of modern Transformer variants would replace the single hidden-state ODE by a coupled system of residual streams with matrix-valued or manifold-constrained connection operators.  The estimates in this paper suggest what must be controlled in such a model: boundedness of the coupled vector field, Lipschitz dependence on states and controls, and stable readout behavior.

Third, stochastic training dynamics can be incorporated more directly.  The current limiting equation is a deterministic transport equation for the hidden-state law.  Adding gradient noise, dropout-type randomness, or explicit diffusion would lead to stochastic-control or Fokker--Planck variants.  This could connect the deterministic mean field control picture with implicit regularization and noisy optimization, while preserving the distinction between depth discretization, sampling, and optimization error.

Fourth, noncompact parameter regimes remain an important open problem.  The present range estimates show how compactness can follow from regularization or Lipschitz-in-time control classes, but large-scale training often uses weaker constraints.  Extending the theory to moment bounds, coercive losses, adaptive normalization, or data-dependent localization would make the framework closer to practical training regimes.

\subsection{Conclusion}

This work provides a rigorous continuous-depth control baseline for Transformer-type residual layers trained by cross-entropy.  Under explicit compactness and smoothness assumptions, finite residual layers approximate a controlled hidden-state flow, empirical cross-entropy risk approximates the population objective, and finite-depth optima can be compared with continuous-control optima.  The resulting mean field control formulation also yields a Pontryagin condition, clarifies the role of the softmax residual in the terminal adjoint, and gives local stability and initialization criteria for finite-depth training.

The main value of the framework is conceptual separation.  Discretization error, sampling error, approximation of the admissible control class, and nonconvex optimization are different issues and require different tools.  By isolating these mechanisms, the paper provides a foundation for sharper generalization estimates, stochastic training limits, noncompact control theory, and more faithful continuous models of full Transformer architectures.




\end{document}